\documentclass[12pt]{article}

\usepackage{amsmath,amsfonts,latexsym,amssymb,amsthm,graphicx}

\usepackage[margin=1in]{geometry}

\usepackage[round]{natbib}
\usepackage[colorlinks,bookmarksopen,bookmarksnumbered,citecolor=red,
urlcolor=red]{hyperref}
\usepackage{tabularx,ragged2e,booktabs,caption}
\usepackage{authblk}
\usepackage{subfig}
\usepackage{setspace}
\usepackage{algorithm}
\usepackage{mathrsfs}

\newcommand{\Ecal}{\mathcal{E}}
\newcommand{\Ecalhat}{\hat {\Ecal}}

\newcommand{\what}{\hat{w}}

\newcommand{\whatlambdan}{\hat{w}_{\lambda}}
\newcommand{\thetahatlambdan}{\hat{\theta}_{\lambda}}
\newcommand{\wstarlambdan}{w^\ast_{\lambda}}
\newcommand{\wstarlambdant}{w^{\ast T}_{\lambda}}

\newcommand{\bfm}[1]{\ensuremath{\mathbf{#1}}}

     \def\bQ_\lambda{\bfm Q_\lambda}     \def\cQ_\lambda{{\cal  Q_\lambda}}

          \def\cX{{\cal  X}}

\renewcommand{\hat}{\widehat}

\def\efn{E_{F_n}}
\def\thetahat{\hat \theta}
\def\uml{u^{ML}}

\newtheorem{theorem}{Theorem}[section]

\newtheorem{corollary}[theorem]{Corollary}

\newtheorem{lemma}[theorem]{Lemma}

\makeatletter
\def\BState{\State\hskip-\ALG@thistlm}
\makeatother

\title{Fast construction of efficient composite likelihood equations}

\author{Zhendong Huang}
\author{Davide Ferrari \thanks{Corresponding author: Davide Ferrari, School of Mathematics and Statistics, The University of
		Melbourne, Parkville, VIC 3010, Australia. E-mail: \url{dferrari@unimelb.edu.au}.}}
\affil{School of Mathematics and Statistics, The University of Melbourne}
\date{}

\begin{document}

\maketitle

\begin{abstract}
Growth in both size and complexity of modern data challenges
the applicability of traditional likelihood-based inference. Composite likelihood (CL) methods address the  
difficulties related to model selection and computational intractability of the full likelihood by combining a number of low-dimensional likelihood objects into a single objective function used for inference.  This paper introduces a procedure to combine partial likelihood objects from a large set of feasible candidates and simultaneously carry out parameter estimation. The new method constructs estimating equations balancing statistical efficiency and computing cost by minimizing an approximate distance from the full likelihood score subject to a $\ell_1$-norm penalty representing the available computing resources. This results in truncated CL equations containing only the  most informative partial likelihood score terms. An asymptotic theory within a framework where both sample size and data dimension grow is developed and finite-sample properties are illustrated through numerical examples.
\end{abstract}


\section{Introduction} \label{sec:intro}

Since the idea of likelihood was fully developed by \cite{fisher1922mathematical}, likelihood-based inference has played a role of paramount importance in statistics. The complexity of  modern data, however, poses nontrivial challenges to traditional likelihood  methods. One issue  is related to  model selection, since the full likelihood function can be difficult or impossible to specify in complex multivariate problems. Another difficulty concerns computing and
the necessity to obtain inferences quickly. These challenges have motivated the development of composite likelihood (CL) methods, which avoid 
intractable full likelihoods by compounding a set
of low-dimensional likelihood objects. \cite{Besag74} pioneered CL inference in the context of spatial data; \cite{lindsay1988composite} developed CL inference in its generality. Due
to its flexible framework and established theory, the CL framework has become a popular tool
in many areas of applied statistics; see \cite{varin2011overview} for an
overview of CL inference and common applications.

Consider $n$ independent observations on the $d\times 1$ random vector
$X=(X_1, \dots, X_d)^T$  with pdf in the parametric family $\{f(x;\theta)$, $x\in \cX$, $\theta
\in \Theta \subseteq \mathbb{R}^p\}$, where  $\theta^\ast \in \Theta$ denotes the true parameter. In this paper, we are mainly concerned with large data sets where both the data dimension $d$ and the sample size $n$ are large. Given i.i.d. observations $X^{(1)}, \dots, X^{(n)}$ on $X$, we write $E_{F_n}(g) = n^{-1}\sum_{i\le n}g(X^{(i)})$ for the empirical mean of the function $g$, where $F_n(x)= n^{-1}\sum_{i\le n}I(X^{(i)}\le x)$ is the empirical cdf, and  use $E(g)$ to denote its expected value. The operator
``$\nabla$'' denotes  differentiation with respect to
$\theta$. In the CL setting, the maximum likelihood score $u^{ML}(\cdot, \theta) = \nabla \log f(\cdot,
\theta)$ and the associated estimating equations $E_{F_n} \uml(\theta) =0$ are intractable  due to difficulties in computing or  specifying the full $d$-dimensional
density $f(\cdot; \theta)$.  Suppose, however, that one can obtain $m$ tractable  pdfs $f_1(s_1;\theta), \dots, f_m(s_m; \theta)$ for
sub-vectors $S_1,\dots, S_m$ of $X$, where each $S_j$ has dimension much
smaller than $d$. For example, 
$S_1$ could represent a single element of $X$ like $X_1$, a variable pair
like $(X_1 , X_2)$, or a conditional sub-vector like $(X_1 , X_2)|X_1$. Typically, the total number of sub-models $m$ 
grows quickly with  $d$; for instance, taking all variable pairs in
$X$ results in $m=d(d-1)/2$ candidate sub-likelihoods. The specific choice for the set of pdfs $\{f_j, j=1,\dots, m\}$ is sometimes referred to as CL design \citep{lindsay2011issues} and is typically specified by the practitioner . For simplicity, here the CL design is treated as given, and we assume $f_1=\dots = f_m$, as it is often the case in applications.

We focus on the maximum composite likelihood estimator  (MCLE),  $\thetahat(w)$, defined as the solution to the CL estimating equations
\begin{equation} \label{estimating_eq}
0 = E_{F_n}[u(\theta, w)]  = \efn [w_1  u_1(\theta)  + \dots + w_m u_m(\theta)],
\end{equation}
where $u_j(\cdot,\theta) = \nabla \log \{ f_j(\cdot; \theta) \}$ is the $j$th partial score (sub-likelihood score) associated with the $j$th subset $S_j$ of $X$. Here $w \in \mathbb{R}^{m}$ is a given vector 
of coefficients to be determined, which we refer to as composition rule. In addition to well-known computational advantages  compared to MLE and flexible modeling, the MCLE  enjoys first-order
properties analogous to those of the maximum likelihood estimator (MLE). Since the partial scores commonly define unbiased estimating equations (i.e. $Eu_j(\theta) =0$ at $\theta=\theta^\ast$, for all $1\leq j \leq m$), the CL score $u(\theta, w)$ in (\ref{estimating_eq}) is also unbiased, a property  leading to consistency of  $\thetahat(w)$.  
Unfortunately, the MCLE does not have the same second-order properties as the MLE since the asymptotic variance of 
$\thetahat(w)$ is generally different from the inverse of 
Fisher information $ - E [\nabla  u^{ML}(\theta^\ast)]$, with the two coinciding only in special families of models.

The choice of the composition rule $w$ is crucial in determining both  efficiency and computing cost associated with $\thetahat(w)$. Established theory of unbiased estimating equations prescribes to find $w$ so to minimize the asymptotic variance of $\thetahat(w)$ \citep[Chapter 2]{heyde2008quasi}, given by the inverse of the $p\times p$ Godambe information matrix
\begin{equation}\label{eq:Godambe}
\mathcal{G}(\theta, w) = E\{\nabla u(\theta, w)\} \ \text{var}\left\{ u(\theta, w) \right\}^{-1} E\{\nabla u(\theta, w)\}.
\end{equation}
Although theoretically appealing,  this is a notoriously difficult  task due the well-known instability of common  estimators of the term  $\text{var}\left\{ u(\theta, w) \right\}$  in $\mathcal{G}(\theta, w)$  \citep{lindsay2011issues}. On the other hand, the common practice of retaining all
terms in (\ref{estimating_eq}) by choosing fixed $w_j \neq 0$ for all $j \geq 1$ (e.g. $w_j=1$,  $j\geq 1$)
is undesirable from both computational and statistical efficiency perspectives, especially when the partial scores $u_j$ exhibit pronounced correlation.   \cite{cox2004note}  discuss the detrimental effect caused by the presence of many correlated scores on the variance of $\thetahat(w)$ when $n$ is small compared to $m$ in  pair-wise likelihood estimation. In the most serious case where the correlation between scores is overwhelming, keeping all the terms in (\ref{estimating_eq}) may lead to lack of consistency for the implied  MCLE $\thetahat(w)$.

Motivated by the above considerations, we introduce a new method called sparse composite likelihood  estimation and selection (SCLE)  consisting of two main steps: a truncation Step (T-Step) and an estimation  Step (E-Step). 
In the T-Step, the composition rule $w$ is obtained by minimizing an approximate distance between the unknown full likelihood score 
$u^{ML}(\theta)$ and the CL score $u(\theta, w)$,  subject to a $\ell_1$-norm constraint on $w$. This step may be viewed as maximizing statistical accuracy for given afforded computing. Alternatively it may be interpreted as minimizing the computing cost for given level of statistical efficiency. Due to the geometry of the $\ell_1$-norm, the resulting composition rule, say $\hat w$, contains a number of non-zero elements (see Lemma \ref{Lem:KKT}). While the most useful terms for improving MCLE's statistical accuracy are retained, the noisy sub-likelihoods contributing little or no improvement are dropped. In the E-step, we solve the estimating equations (\ref{estimating_eq}) with $w=\hat w$ and find the final estimator $\hat \theta (\hat w)$.  Compared to traditional CL estimation, the main advantage of our approach is to reduce the computational burden, while retaining relatively high efficiency in large data sets. The reduced number of terms in the estimating equations (\ref{estimating_eq}) translates into fast computing and enhanced stability for the final estimator at a relatively small cost in terms of statistical efficiency.

The remainder of this paper is organized as follows. In Section \ref{sec:method}, we describe the main methodology for simultaneous likelihood truncation and  parameter estimation. In Section \ref{sec:properties}, we study the properties of the truncated composition rule and for the implied estimator within a framework where both the sample size $n$ and the data dimension $d$ are allowed to diverge. Section \ref{sec:example} illustrates the properties of our methodology in the context of estimation of location and scale for multivariate normal models. In Section \ref{sec:NumExample}, we study the trade-off between computational and statistical efficiency in finite samples through numerical simulations. Section \ref{sec:finalremarks} concludes with final remarks. Technical lemmas used in our main results are deferred to the appendix.

\section{Main methodology} \label{sec:method}

Throughout the paper, we consider unbiased  partial scores $\{u_j(\theta), 1 \le j \le m\}$ satisfying
\begin{equation} \label{eq:sub_equations}
Eu_j(\theta)=0,  \ \text{for all } \ 1 \le j \le m,
\end{equation}
when $\theta = \theta^\ast$ and assume that $\theta^\ast$ is the unique solution for all the  equations in  (\ref{eq:sub_equations}). The approach described in this section is applicable to problems with arbitrary sample size $n$ and data dimension $d$, but we are mainly concerned with the situation where the data dimension $d$ (and number of available sub-likelihood objects $m$) is large compared to the sample size $n$. Although we focus on log-likelihood partial scores for concreteness,  our methodology and the properties in Section \ref{sec:properties} remain essentially unchanged if $u_j(\theta)$ is any arbitrary unbiased M-estimating equation. For instance, when $\theta$ is a location parameter, a more appropriate choice in the presence of outliers may be the Huber-type partial score  $u_j(\theta)= \psi(s_j - \theta)$, where $\psi(z)=-k$ if $z \leq k$,  $\psi(z)=z$ if $|z| \leq k$ and $\psi(z)=k$ if $z\geq k$, with $k>0$. Another suitable choice in the same setting is the Lq-likelihood estimating equation of \cite{ferrari2010maximum} defined by $u_j(\theta) = \nabla \log_q \{f_j(s_j; \theta)\}$, where $\log_q(z)= \log(u)$ if $q=1$, and $\log_q(z)= (z^{1-q}-1)/(1-q)$ if $q\neq 1$.

In the rest of the paper we use $U(\theta)$ to denote the $p \times m$ matrix with column vectors  $u_1(\theta), \dots , u_m(\theta)$ and define the $m \times m$ matrix $S(\theta)= U(\theta)^T U(\theta)$ with  $(jk)$th entry $\{S(\theta)\}_{j,k}=u_j(\theta)^T  u_{k}(\theta)$.  We write $U_{\mathcal{A}}(\theta)$ for the sub-matrix of $U(\theta)$  with columns corresponding to $\mathcal A \subseteq \{1,\dots, m\}$, while $U_{\setminus \mathcal{A}}(\theta)$ denotes the sub-matrix containing the remaining columns. Accordingly, we define the $|\mathcal{A}|\times |\mathcal{A}|$ matrix $S_{\mathcal{A}}(\theta)= U(\theta)_{\mathcal{A}}^T U_{\mathcal{A}}(\theta)$ and use $w_{\mathcal{A}}$ to denote the sub-vector of $w$ with elements $\{w_j$, $j \in \mathcal{A} \}$, while $w_{\setminus \mathcal{A}}$ represents the vector containing all the elements in $w$ not in $w_{\mathcal{A}}$.

\subsection{Sparse and efficient estimating equations} \label{Sec:sparse&efficient}

Our main objective is to solve the CL estimating equations $0=E_{F_n}u(\theta, w)$ defined in (\ref{estimating_eq}) with respect to $\theta$ using coefficients $w=w_\lambda(\theta)$  obtained by minimizing the ideal  criterion
\begin{align} \label{eq:criterion}
Q_\lambda(\theta, w) = \dfrac{1}{2}E \left\|   \uml  (\theta)-  \sum_{j=1}^m w_j u_j(\theta) 
\right\|^2_2   + \lambda  \sum_{j=1}^m \alpha_j \left\vert   w_j \right\vert,
\end{align}
where $\Vert \cdot \Vert_2$ denotes the Euclidean norm, $\lambda 
\geq 0$ is a given constant, and  the $\alpha_j$s are pre-set constants not depending on the data. For clarity of exposition, we  set
$\alpha_j=1$ for all $j \geq 1$ in the remainder of the paper.  The optimal solution $w_\lambda(\theta)$ is interpreted as one that maximizes  the
statistical accuracy of the implied CL estimator, subject to a given level of  computing.  Alternatively, $w_\lambda(\theta)$ may be viewed as to minimize the complexity of the CL equations, subject to given efficiency 
compared to MLE. The tuning constant $\lambda$ balances the trade-off between  statistical efficiency and computational burden

The  first term  in $Q_\lambda(\theta,w)$ aims to obtain efficient estimating equations by
finding a CL score close to the ML  score.  When $\lambda = 0$ and $\theta = \theta^\ast$, the composition rule $w^\ast_0= w_0(\theta^\ast)$ is optimal in the sense that the score function $u(\theta, w^\ast_0)$ is closest to the MLE score $u^{ML}(\theta)$. Although this choice gives estimators with good statistical efficiency, it offers no control for the CL score complexity since all the partial likelihood scores are included in the final estimating equation. The second term $\lambda  \sum_{j=1}^m \alpha_j \left\vert   w_j
\right\vert$ in (\ref{eq:criterion}) is a penalty discouraging overly complex estimating equations. In Section \ref{sec: theory1}, we show that typically this form of penalty implies a number of  elements in  $w_{\lambda}(\theta)$ exactly zero for any $\lambda>0$.  For relatively large $\lambda$, many elements in $w_\lambda(\theta)$ are exactly zero, thus simplifying considerably the CL estimating equations $0=E_{F_n}u(\theta, w_\lambda(\theta))$. When a very large fraction of such elements is zero,  we say that $w_{\lambda}(\theta)$ and the 
CL equations $0=E_{F_n}u(\theta, w_\lambda(\theta))$ are sparse. Sparsity is a key advantage of our approach to reduce the computational burden when achievable without loosing much statistical efficiency. On the other hand, if $\lambda$ is too large, one risks to miss the information in some useful data subsets  which may otherwise improve statistical accuracy.

\subsection{Empirical criterion and one-step estimation}
\label{sec:twostep}

Obvious difficulties related to direct minimization of the ideal criterion  $Q_\lambda(\theta,w)$ are the presence of the intractable 
likelihood score $\uml$ and the expectation depending on the unknown parameter $\theta^\ast$. To address these issues, first note that, up to a negligible term not depending on $w$, Criterion (\ref{eq:criterion}) can be written as 
\begin{equation}\label{eq:crit2}
\dfrac{1}{2} E \left\| \sum_{j=1}^m w_j 
u_j(\theta) \right\|^2_2  
-  \sum_{j=1}^m w_j E \left[
u^{ML}(\theta)^T u_j(\theta)  \right]+ \lambda  \sum_{j=1}^m \alpha_j
\left\vert   w_j \right\vert.
\end{equation}
If $\theta=\theta^\ast$, we have $E[\uml(\theta) 
u_j(\theta )^T] = E[u_j(\theta )u_j(\theta )^T]$. To see this, recall that partial scores are unbiased and   differentiate both sides of $0 = Eu_j(\theta)$ under appropriate regularity conditions.  This result is used to  eliminate the explicit dependency on the score $\uml$. Finally, replacing expectations in (\ref{eq:crit2}) by empirical averages leads to the following empirical objective:
\begin{align} \label{eq:objective_empirical}
\hat Q_\lambda(\theta, w) = \dfrac{1}{2}\efn \left\| \sum_{j=1}^m w_j 
u_j(\theta) \right\|^2_2 -  \sum_{j=1}^m w_j   \efn\left[
u_j(\theta)^T u_j(\theta)  \right]+ \lambda  \sum_{j=1}^m \alpha_j
\left\vert   w_j \right\vert.  
\end{align}

Under appropriate regularity conditions, the empirical criterion (\ref{eq:objective_empirical}) estimates consistently the population criterion (\ref{eq:criterion}) up to an irrelevant constant not depending on $w$, with the caveat that $\theta$ must be close to  $\theta^\ast$. These considerations motivate the following estimation strategy:
\begin{itemize}
	\item[1)] {\it T-Step.} Given a preliminary root-$n$ consistent estimator $\thetahat$,  compute the truncated composition rule $\what_\lambda$ by solving 
	\begin{equation} \label{eq:LAstep}
	\hat w_\lambda = \underset {w \in \mathbb{R}^m} {\text{argmin}}   \ \hat Q_\lambda (\thetahat, w).
	\end{equation}
	\item[2)] {\it E-Step.} Update the parameter estimator by the one-step Newton-Raphson iteration
	\begin{equation} \label{eq:Estep}
	\thetahat_\lambda =\thetahat -\left
	[\efn\nabla u(\thetahat ,\what_\lambda)\right ]^{-1}
	\efn u(\thetahat ,\what_\lambda).
	\end{equation}
\end{itemize}
Theorem \ref{Thm:uniqueroots} shows that the convex minimization problem in the T-Step has unique solution. Particularly,  let ${\Ecalhat} \subseteq \{1,\dots, m\}$ is the subset of partial scores such that
\begin{equation}\label{eq:cov_with_mle}
\left|  \efn \left\{ u_j(\thetahat)^T r_j(\thetahat, \what_\lambda) \right\} \right| \geq \lambda,
\end{equation}
where $r_j$ is the pseudo-residual defined by $r_j(\theta, w) =  u_j(\theta)- u(\theta, w)$ and 
and write $\setminus {\Ecalhat}$ for the set $\{1,\dots m\} \setminus {\Ecalhat}$. Then the solution  of the T-Step  is
\begin{equation} \label{eq:whatsol_emp}
\hat w_{\lambda,  \Ecalhat}  =  \left\{ \efn  S_{\Ecalhat}(\thetahat)   \right\}^{-1} \left\{ \text{diag}\{\efn S_{ \Ecalhat} (\thetahat)\} - \lambda  \ \text{sign}(\hat w_{\lambda,  \Ecalhat } ) \right\}, \ \ \hat w_{\lambda, \setminus  \Ecalhat } = 0, 
\end{equation}
where: $S_{\Ecalhat}= U^T_{\Ecalhat} U_{\Ecalhat}$ and $U_{\Ecalhat}$ is a matrix with column vectors $\{ u_j, j \in \Ecalhat \}$; $\text{sign}(w)$ is the vector sign function with $j$th element taking values $-1$, $0$ and $1$ if $w_j<0$, $w_j=0$ and $w_j>0$, respectively; and $\text{diag}(A)$ denotes the diagonal of the square matrix $A$.

More insight on the meaning of (\ref{eq:cov_with_mle})  may be useful.  Differentiating (\ref{eq:crit2}) in $w_j \neq 0$  and then expanding around $\theta^\ast$ under Conditions C.1 and C.2  in Section \ref{sec: theory1} gives
\begin{equation}\label{eq:cov_with_mle2}
\efn \left\{ u_j(\thetahat)^T r_j(\thetahat , w) \right\}   =  E\left\{ u_j(\theta^\ast)^T \left[ u^{ML}(\theta^\ast) - u(\theta^\ast, w) \right] \right\} + o_p(1).
\end{equation}
Combining (\ref{eq:cov_with_mle}) and (\ref{eq:cov_with_mle2}) highlights that the $j$th partial likelihood score $u_j(\theta)$ is selected when it is sufficiently correlated with the residual difference $u^{ML}(\theta) - u(\theta, w)$. Hence, our criterion retains only those $u_j$s which are  maximally useful to explain the gap between the full likelihood score $u^{ML}(\theta)$ and the CL score $u(\theta, w)$, while it drops the remaining scores.

When $\lambda= 0$, we have
${\Ecalhat} = \{1,\dots, m\}$ meaning that the corresponding composition rule $\what_0$ does not contain zero elements. From (\ref{eq:whatsol_emp}) for $\lambda=0$ it is required that the empirical covariance matrix for all partial scores $\efn  S(\thetahat)$  is non-singular  which is violated when $n<m$. Even for $n>m$, however, $\efn  S(\thetahat)$ may be nearly singular due to the presence of largely correlated  partial scores. On the other hand,  setting $\lambda > 0$ always gives a non-singular matrix $\efn S_{\Ecalhat}(\thetahat)$ and guarantees existence of $\hat w_{\lambda,  \Ecalhat}$. 

The proposed approach requires an initial root-$n$ consistent estimator, which is often easy to obtain when the partial scores are unbiased. One simple
option entails solving $\efn u(w,\theta)=0$ with $w=(1, \dots, 1)^T$. 
If $m$ is  large, one may choose $w$ by the    
stochastic CL strategy of
\cite{dillon2010stochastic}, where the elements of $w$ may be set as either 0
or 1 randomly
according to some user-specified scheme. Although the initial estimator $\thetahat$ could
be quite inefficient, the one-step update (\ref{eq:Estep}) 
improves upon this situation. 
Moreover, the estimator $\thetahat_\lambda$  and coefficients
$\what_\lambda$ can be refined by iterating the T-Step (with
$\thetahat = \thetahat_\lambda$) and the E-Step a few times.

\subsection{Computational aspects: LARS implementation and selection of $\lambda$} \label{Sec:lambda}

The empirical composition rule $\what_\lambda$ in (\ref{eq:LAstep}) cannot be computed using derivative-based 
approaches due to non-differentiability of $\hat Q_\lambda(\thetahat, w)$. To address this issue, we propose an implementation based on the
least-angle regression (LARS) algorithm of \cite{Efron04} originally developed for sparse parameter estimation in the context of linear regression models. For given $\theta=\thetahat$, our implementation of LARS  minimizes $\hat 
Q_\lambda(\thetahat, w)$ by including one score  $u_j(\thetahat)$ at the time in the composite likelihood score $u(\thetahat, w)$.  In each step, the score with the largest correlation with the currently available residual difference 
$u_j(\thetahat) - u(\thetahat, w)$ is included, followed by an adjustment step on $w$.  The numerical examples in Section \ref{sec:NumExample},  suggest that our implementation of the LARS algorithm for CL selection is very fast. In at most $m \times p$ steps, it returns a path of estimated composition rules  $\what_{\lambda_1}, \dots, \what_{\lambda_m}$, where $\lambda_j$ here is the value of the tuning constant $\lambda$ in (\ref{eq:objective_empirical}) at which the $j$th partial score  enters the CL estimating equation.

Selection of $\lambda$ is of practical importance since it balances the trade-off between statistical and computational efficiency. For a given  budget on afforded computing, say  $\lambda^\ast$,  we  include one partial score at the time, for example using the LARS approach above,  and stop when we reach $\hat{\lambda}=\max \{   \lambda: \phi(\lambda) > \tau \}$, for some user-specified $0 <\tau \leq 1$, where 
\begin{equation}\label{eq:lambda}
\phi(\lambda) = \frac{   \text{tr}\{\efn S_{\lambda} (\thetahat)\}      }{     \text{tr}\{\efn S (\thetahat)\}          }I(\lambda >\lambda^\ast).
\end{equation}
Here $\efn S_{\lambda}= \efn U^T_{\lambda} U_{\lambda}$ denotes the empirical covariance matrix for the selected partial scores indexed by the set $\Ecalhat_\lambda = \{ j: \what_{\lambda,j} \neq 0 \}$. The criterion $\phi(\lambda)$ can be viewed as the proportion of score variability  explained by currently selected partial scores.  In practice,  we choose $\tau$ close to 1, such as $\tau = 0.9$, $0.95$ or $0.99$. If the computing budget is reached, we set $\hat \lambda = \lambda^\ast$. In analogy with principal component analysis, the selected combination of scores accounts for the largest variability in the collection of empirical scores.

\section{Properties}\label{sec:properties}

This section investigates the asymptotic behavior of the sparse composition rule $\what_\lambda$ and the corresponding SCLE $\thetahatlambdan$ defined in (\ref{eq:Estep}) within a setting where  $m$ -- the number of candidate partial likelihoods -- is allowed to grow with the sample size $n$. We use $m^\ast=E\|u^{ML}(\theta^\ast)\|_2^2$ to denote the trace of Fisher information based on the full likelihood. Here $m^\ast$  may be interpreted as the  maximum  knowledge about $\theta$ if the full likelihood score $u^{ML}$ were available. Although $m^\ast$ can  grow with $m$, reflecting the rather natural  notion the one can learn more about the true model as the overall data size increases, it is not allowed to grow as fast as $n$; e.g., $m^\ast=o(\log n)$.  This is a rather common situation in CL estimation occuring, for instance, when the sub-likelihood scores are substantially correlated or they are independent but with heterogeneous and increasing variances (see examples in Section \ref{sec:anaexample1}).

\subsection{Sparsity and optimality of the composition rule} \label{sec: theory1}

In this section, we  give conditions ensuring uniqueness of the empirical composition rule $\what_{\lambda}$ and  weak convergence  to its population counterpart $\wstarlambdan$.  To this end, we work $\theta$ within the root-$n$ neighborhood of $\theta^\ast$, $\Theta_n=\{\theta: \Vert \theta - \theta^\ast \Vert < c_0  {n}^{-1/2}\}$, for some $c_0>0$, and assume the following regularity conditions on  $S(\theta)$:
\begin{itemize} 
	\item[C.1]\label{cond_c1} There exist positive constants $c_1$, $c_2>0$ such that  $E\{\sup_{\theta \in \Theta_n} S(\theta)_{j,k}\}<c_1$, and $Var\{\sup_{\theta \in \Theta_n} S(\theta)_{j,k}\}<c_2$, for all $j,k \geq 1$.
	\item[C.2]\label{cond_c2} Each element $ES(\theta)_{j,k}$ is continuous with uniformly bounded first and second order derivatives on $\Theta_n$.
\end{itemize}
Our analysis begins by deriving the Karush-Kuhn-Tucker Condition (KKTC) \citep{kuhn2014nonlinear} for the population objective $ Q_\lambda(\theta^\ast, w)$ defined in (\ref{eq:criterion}). The KKTC characterizes the amount of sparsity -- and, the computational complexity -- associated with the selected estimating  equations  depending on the value of the tuning constant $\lambda$.  Let $c(\theta,w)=\text{diag}(S(\theta))-S(\theta)w$, where $S(\theta)= U(\theta)U(\theta)^T$ is as defined in Section \ref{sec:method} 

\begin{lemma}[KKTC] \label{Lem:KKT} Under Condition  \emph{C.1}, the minimizer $w^\ast_\lambda$ of $Q_\lambda(\theta^\ast, w)$ defined in (\ref{eq:criterion}) 
	satisfies
	\begin{equation*}
	E \vert \{c(\theta^\ast,   w^\ast_\lambda)_j\} \vert =\lambda\cdot \gamma_j, \ \  j = 1, \dots, m,  
	\end{equation*}
	where $\gamma_j \in \{1\}$ if $w^\ast_{\lambda,j} >0$, $\gamma_j \in \{-1\}$ if $w^\ast_{\lambda,j} <0$, and  $\gamma_j \in [-1, 1]$ if 
	$w^\ast_{\lambda,j} = 0$;  $c(\cdot,   \cdot)_j$ is the $j$th element of vector $c(\cdot,   \cdot)$.
\end{lemma}

\begin{proof} Let
	$
	d_j = 
	- E c(\theta^\ast,w^\ast_\lambda)_j+\lambda \cdot \text{sign}(w^\ast_{\lambda, j})
	$ and note that the Tayor expansion of $Q_\lambda(\theta^\ast, w^\ast_\lambda)$ around $w^\ast_{\lambda, j}\neq 0$ is
	\begin{equation}\label{eq:expansion}
	Q_{\lambda}(\theta^\ast,w^\ast_\lambda+\epsilon)=  Q_{\lambda}(\theta^\ast, w^\ast_\lambda)+\epsilon d_j +  \dfrac{{\epsilon}^2}{2} \text{tr}\{I_j(\theta^\ast)\},
	\end{equation}
	where $\epsilon=(0,\dots,   \epsilon_j,\dots,0)^T$, and $I_j(\theta^\ast) =  E\left[ 
	u_j(\theta^\ast)u_j(\theta^\ast)^T \right]$ is the $p\times p$ Fisher information matrix for the $j$th likelihood component, and 
	$\text{tr}\{ I_j(\theta^\ast) \}< c_1$ by Condition 
	C.1. 
	
	If $w^\ast_{\lambda,j}  \neq 0$, we have $d_j =0$. Otherwise, if $d_j \neq 0$, choosing $\epsilon_j$ such that $\text{sign}(\epsilon_j)= - \text{sign}(d_j)$ and $|\epsilon_j| < 2|d_j|/c_1$, 
	implies $Q_{\lambda}(\theta^\ast,w^\ast_\lambda+\epsilon)<  Q_{\lambda}(\theta^\ast,w^\ast_\lambda)$,  but this is a contradiction 
	since $w^\ast_\lambda$ minimizes $ Q_{\lambda}(\theta^\ast,\cdot)$.  If $w^\ast_{\lambda,j} = 0$, we need to show $| Ec(\theta^\ast,w^\ast_\lambda)_j | \leq \lambda$. Assume $| Ec(\theta^\ast,w^\ast_\lambda)_j | > \lambda$ and take $\epsilon_j$ such that $\text{sign}(\epsilon_j)= \text{sign}(Ec(\theta^\ast, w^\ast_\lambda)_j)$ and $|\epsilon_j|< 2|Ec(\theta^\ast, w^\ast_\lambda)_j-\lambda|/|c_1|$.  Then $\epsilon d_j + \epsilon^2 \text{tr}\{I(\theta^\ast)\}/2 < -|\epsilon_j| (|Ec(\theta^\ast, w^\ast_\lambda)_j|-\lambda) + \epsilon^2 c_1/2<0$, which implies $Q_\lambda(\theta^\ast,w^\ast_\lambda + \epsilon)<Q_\lambda(\theta^\ast,w^\ast_\lambda)$. But this is contradicted by  $w^\ast_\lambda$ being the minimizer of $Q_\lambda$. Hence, $E c(\theta^\ast,   w^\ast_\lambda)_j=\lambda\cdot \gamma_j$, for all $j=1,\dots,m$.  
\end{proof}

An argument analogous to that used in the proof of Lemma \ref{Lem:KKT} leads to the KKTC for $\what_\lambda$, the minimizer of the empirical loss $\widehat{Q}(\thetahat,w)$. Specifically, for $\what_\lambda$ we have $
\efn c(\thetahat,   \what_\lambda)_j=\lambda\cdot \hat{\gamma}_j, \ \  j = 1, \dots, m  
$, where $\hat{\gamma}_j \in \{1\}$ if $\what_{\lambda,j} >0$, $\hat{\gamma}_j \in \{-1\}$ if $\what_{\lambda,j} <0$, and  $\hat{\gamma}_j \in [-1, 1]$ if 
$\what_{\lambda,j} = 0$.

Lemma \ref{Lem:KKT} has important implications in our current setting, since it relates $\lambda$ to the size of the covariance between the $j$th sub-likelihood 
score $u_j(\theta)$ and the residual difference $u^{ML}(\theta)-u(\theta,w)$ at $\theta = \theta^\ast$. Particularly, if such a covariance is sufficiently small, i.e.
\begin{align*}
\lambda> E\{c(\theta^\ast,   w^\ast_\lambda)_j \} = \left|  E\left\{ u_j(\theta^\ast)\left[ u(\theta^\ast, w^\ast_\lambda) - u_j(\theta^\ast) \right] \right\} \right| 
= \left|  E\left\{ u_j(\theta^\ast)\left[ u(\theta^\ast, w^\ast_\lambda) - u^{ML}(\theta^\ast) \right] \right\} \right| ,
\end{align*}
then the correspondent coefficient is $w^\ast_{\lambda,j} = 0$. Thus, the tuning parameter $\lambda$ controls the level of sparsity of the composite score $u(\theta^\ast, w^\ast_\lambda)$ by forcing the weights of those non-important score components with small pseudo-covariance $c(\theta^\ast,w^\ast_\lambda)_j$ to be exactly zero.

For uniqueness of  $w^\ast_\lambda$ and $\what_\lambda$,   a simple condition is that the partial scores cannot replace each other, i.e. we  require that the scores are in general position. Specifically, we say that the score components $u_1,\dots,u_m$ are in general position if any affine subspace $\mathbb{L} \subset \mathbb{R}^m$ of dimension $l<m$ contains at most $l+1$ elements of $\{ \pm u_1,...,\pm u_m \}$  excluding antipodal pairs of points. 
\begin{itemize}
	\item[C.3] The partial scores $u_j(x,\theta)$, $j\geq 1$, are continuous and in general position with probability 1 for all $\theta \in \Theta_n$.
\end{itemize}

\begin{theorem} \label{Thm:uniqueroots}
	Under Conditions \emph{C.1}-\emph{C.3} the solution of the T-Step,  $\what_\lambda$, defined in (\ref{eq:LAstep}) is unique and is given by (\ref{eq:whatsol_emp}) for any $ \lambda>0$. Moreover, $\what_\lambda$ contains at most $np \wedge m$ non-zero elements.
\end{theorem}

\begin{proof} Let  $\hat\Ecal=\{ j \in \{1, \dots, m\}: | \hat\gamma_j|= 1 \}$ to be the index set of non-zero elements of  $\what_\lambda$ where  where $\gamma _j$ is as defined after Lemma \ref{Lem:KKT}.  First note that the composite likelihood score $u(\thetahat,\what_\lambda)=U(\thetahat)^T  \what_\lambda$ is unique  for all solutions $\what_\lambda$ which minimize  $ \hat Q_\lambda(\thetahat,\cdot)$ defined in (\ref{eq:objective_empirical}), due to  strict convexity of $\hat Q_\lambda(\thetahat,w)$. Uniqueness of $u(\thetahat, \what_\lambda)$ implies that $\hat\gamma$ and the corresponding index set $\hat\Ecal$ are unique by Lemma \ref{Lem:KKT}.  
	
	Next, to show uniqueness of  $\what_{\lambda,j}$ for all $j \in { {\hat\Ecal}}$, we first note that the square matrix $\efn[U_{ {\hat\Ecal}}(\thetahat)^T U_{ {\hat\Ecal}}(\thetahat)]$ has full rank. Otherwise, some row the matrix can be written as a linear combination of other rows in the set ${ {\hat\Ecal}}$, i.e. $\efn[u_{k}(\thetahat)^T U_{ {\hat\Ecal}}(\thetahat)] =\sum_{k\neq j}a_j \efn[u_j(\thetahat)^T U_{ {\hat\Ecal}}(\thetahat)] $. Then Lemma \ref{Lem:KKT} implies also the event $\efn u_k(\thetahat)^2 - \sum_{j\neq k} a_j \efn u_j(\thetahat)^2 = \lambda \hat\gamma_k - \sum_{j\neq k} a_j\lambda \hat\gamma_j$ for the same set of coefficients $a_j$s, which has probability equals to 0 since each $u_j$ is continuous and random.  Thus, $E[U_{ {\hat\Ecal}}(\thetahat) U_{ {\hat\Ecal}}(\thetahat)^T]$ has full rank, meaning that the size of ${ {\hat\Ecal}}$ satisfies $|{ {\hat\Ecal}}| \leq np \wedge m$.  For fixed  $w_j=0$, $j \in \setminus { {\hat\Ecal}}$,   full rank of $E[U_{ {\hat\Ecal}}(\thetahat) U_{ {\hat\Ecal}}(\thetahat)^T]$  implies strict convexity of  $ {\hat Q}_\lambda(\thetahat,w_{\lambda, \hat\Ecal})$ where $w_{\lambda, \hat\Ecal}$ is the sub-vector of $w$ containing elements indexed by $\hat\Ecal$.  Hence, $\what_{\lambda,\hat\Ecal}$ is unique.  \end{proof}

The arguments in Theorem \ref{Thm:uniqueroots} go through essentially unchanged  for the population composition rule $w^\ast_\lambda$ by showing the full rank of $E[U_{ {\Ecal}}(\theta^\ast)^T U_{ {\Ecal}}(\theta^\ast)]$ using Condition C.3 and Lemma \ref{Lem:KKT}, where $\Ecal$ is the index set of non-zero elements in $w^\ast_\lambda$. This implies also uniqueness of $w^\ast_\lambda$.  Next, we turn to convergence of the empirical composition rule $ \what_\lambda$ to $w^\ast_\lambda$, thus showing that the objective 
$\widehat{Q}_\lambda(\theta,w)$  (\ref{eq:objective_empirical}) is a suitable replacement for the intractable criterion $Q_\lambda(\theta,w)$ (\ref{eq:criterion}). Since  criterion 
$$
\hat Q_\lambda(\thetahat, w)=   w^T \efn \{S(\thetahat) \}w /2-\efn \{\text{diag}(S(\thetahat))\}^T w +\lambda\| w\|_1 
$$ 
is used as an approximation of the population criterion $Q_\lambda(\theta^\ast,w)$ defined in (\ref{eq:criterion}), clearly the distance between $\efn S(\thetahat)$ and $ES(\theta^\ast)$ affects the accuracy of such an approximation. Let $r_1=\sup_{\theta \in \Theta_n}\|  \efn S(\theta) -E S(\theta^\ast)  \|_2$  be the supreme variation between matrices $\efn S(\thetahat)$ and $ES(\theta^\ast)$, where $\| A \|_2$ is the matrix induced 2-norm for matrix $A$.   As $n \to \infty$, the rate at which $r_1$ goes to 0 depends mainly on the number of partial scores $m$ and the behavior of the random elements in $S$, which can vary considerably in different models. For example, when the elements of $S(\theta)$ are sub-Gaussian,  one needs only $\log(m)/n=o(1)$ \citep{cai2010optimal}. In more general cases, $m^4/n=o(1)$ suffices to ensure $r_1 = o_p(1)$ \citep{vershynin2012close}. 

Next we investigate how $m$ and ${m^\ast}$ should increase compared to $r_1$ when $\lambda \rightarrow 0$ as $n\rightarrow \infty$  to ensure a suitable behavior for $\what_\lambda$.  To obtain weak convergence of  $\what_\lambda$ to $w^\ast_\lambda$, we  introduce the additional requirement that the covariance matrix of the partial scores $ES(\theta^\ast)$ does not shrink to zero too fast. 
\begin{itemize}
	\item[C.4] There exists a sequence $c_n$, such that $c_n\frac{\lambda^2}{r_1{m^\ast}^2} \to \infty$ and $x^T ES(\theta^\ast) x \geq c_n \| x \|_1$ for any  $x \in \mathbb{R}^m$, as $n \rightarrow \infty$.
\end{itemize}
Condition C.4 is analogous to the  compatibility condition in $\ell_1$-penalized least-squares estimation for regression  \citep{buhlmann2011statistics}, where it ensures a good behavior of the observed design matrix of regressors. Differently from the sparse regression setting, where Condition C.4 is applied to the set of true nonzero regression coefficients, here no sparsity assumption on the composition rule $w$ is imposed. 

\begin{theorem} \label{Thm:unifweights}
	Under Conditions \emph{C.1}-\emph{C.4}, if $ r_1{m^\ast}^2  \lambda^{-2}=o_p(1)$ then $\| \what_\lambda - \wstarlambdan \|_1 \overset{P}{\to} 0$,
	as $n\to\infty$. 
\end{theorem}
\begin{proof}
	From Lemma \ref{Lem:covergeinuthetahat}, $\efn \|  U(\thetahat) (\what_\lambda-w^\ast_\lambda)  \|_2^2 = o_p(1)$. Note that
	\begin{align*} 
	&\efn \|  U(\thetahat) (\what_\lambda-w^\ast_\lambda)  \|_2^2= (\what_\lambda-w^\ast_\lambda)^T \efn S(\thetahat) (\what_\lambda-w^\ast_\lambda)\\
	=& (\what_\lambda-w^\ast_\lambda)^T ES(\theta^\ast) (\what_\lambda-w^\ast_\lambda) + (\what_\lambda-w^\ast_\lambda)^T \left\{ \efn S(\thetahat)  - ES(\theta^\ast)  \right\} (\what_\lambda-w^\ast_\lambda),
	\end{align*}
	and  the second term of the last equality is $O_p(r_1{m^\ast}^2\lambda^{-2})$ by Lemma \ref{Lem:wnorm}. Thus, $(\what_\lambda-w^\ast_\lambda)^T ES(\theta^\ast) (\what_\lambda-w^\ast_\lambda) = O_p(r_1{m^\ast}^2\lambda^{-2})$, which implies $\| \what_\lambda - \wstarlambdan \|_1 \overset{P}{\to} 0$ by Condition C.4.    \end{proof}

\begin{corollary} \label{Col:covergeinu} Let $\lambda$ be a sequence such that $\lambda \rightarrow 0$ as $n \rightarrow \infty$. 
	Under Conditions \emph{C.1}-\emph{C.4}, if $r_1{m^\ast}^2\lambda^{-2}=o_p(1)$, we have
	\begin{equation}\label{f5}
	\sup_{\theta \in \Theta_n}  \efn  \left 	\|   u(\theta,\what_\lambda)-    u(\theta,\wstarlambdan) \right\|_2 \overset{P}{\to} 0, \quad \text{as $n\ \to \infty$.}
	\end{equation}
\end{corollary}

\begin{proof}
	From Lemma \ref{Lem:covergeinuthetahat},  $ \efn   	\|   u(\thetahat,\what_\lambda)-    u(\thetahat,\wstarlambdan)  \|_2^2 =  o_p(1)$. The result follows by noting that for any $\theta \in \Theta_n$, the difference $\efn \|   u(\theta,\what_\lambda)-    u(\theta,\wstarlambdan)  \|_2^2- \efn\|   u(\thetahat,\what_\lambda)-    u(\thetahat,\wstarlambdan)  \|_2^2 =   (\what_\lambda-w^\ast_\lambda)^T \efn[S(\theta) -S(\thetahat)]  (\what_\lambda-w^\ast_\lambda)$ is $o_p(1)$ according to Conditions C.1, C.2 and Theorem \ref{Thm:unifweights}
\end{proof}

Corollary \ref{Col:covergeinu} states that the composite likelihood score $u(\theta,\what_\lambda)$ is a reasonable approximation to $u(\theta,w^\ast_\lambda)$. Particularly, even for $\lambda$ close to zero, the composite score $u(\theta,\what_\lambda)$ still uses a fraction $|\Ecalhat|/m$ of sub-likelihood components. At the same time $u(\theta,\what_\lambda)$ is near the optimal score $u(\theta,w^\ast_0)$,  where $w^\ast_0$ is the  composition rule yielding the closest CL score $u(\theta)$ to the maximum likelihood score $u^{ML}(\theta)$. Moreover, the implied Godambe information
$$
\mathcal{G}(\theta, \what_\lambda) = E \{ \nabla u(\theta, \what_\lambda) \}
var\{ \nabla u(\theta, \what_\lambda) \}^{-1}
E \nabla \{u(\theta, \what_\lambda) \}
$$
is expected to be close to $\mathcal{G}(\theta, w)$ with $w=w^\ast_0$. However, while the MCLE based on   
$w^\ast_0$ (or other choice of $w_j \neq 0$, $j\geq 1$)  may be unavailable or computationally intractable due to common difficulties in estimating $var\{ \nabla u(\theta, w^\ast_0) \}$ \citep{lindsay2011issues, varin2011overview}, our truncated composition rule $\whatlambdan $ implies a more stable estimation of $\mathcal{G}(\theta, \what_\lambda)$ by requiring only a fraction of scores.

\subsection{Asymptotic behavior of the one-step SCLE}

In this section, we show consistency and give the asymptotic distribution for the SCLE $\thetahatlambdan = \thetahat(\whatlambdan)$ defined in the E-Step (\ref{eq:Estep}). One advantage of one-step estimation is that consistency and asymptotic normality are treated separately. The one-step estimator $\thetahat_\lambda$ inherits the properties leading to consistency from the preliminary estimator $\thetahat$, under standard requirements on $S(\theta)$.   For normality, additional conditions on the sub-likelihood scores are needed. Let $H(\theta)$ be the $p \times mp$ matrix obtained by stacking all the $p \times p$
sub-matrices $\nabla u_j(\theta)$. Let $r_2=\sup_{\theta \in \Theta_n} \max_{j,k}| \efn H(\theta)_{j,k} -E H(\theta^\ast)_{j,k}  |$ be the maximum variation between the empirical and the optimal Hessian matrices.  Let $r_3=\sup_{\theta \in \Theta_n} \max_{j}\| \efn u_j(\thetahat) -E u_j(\theta^\ast)  \|_1$ be the supreme variation between empirical scores and their expected value around $\Theta_n$.  In the rest of this section, we use  $J^\ast_\lambda = Cov \left [u(\theta^\ast,\wstarlambdan) \right ]$ and $K^\ast_\lambda = -E \nabla u(\theta^\ast,\wstarlambdan)$ to denote the population variability and sensitivity $p \times p$  matrices, respectively, both depending implicitly on $n$. We further assume:
\begin{itemize}
	\item[C.5] There exist positive constants $c_5$ and $c_6$ such that $E[\sup_{\theta \in \Theta_n} H(\theta)_{j,k}]<c_5$, and\\ $Var[\sup_{\theta \in \Theta_n} H(\theta)_{j,k}]<c_6$, for all $j,k\geq 1$.          
	\item[C.6]  Each element $EH(\theta)_{j,k}$, $j,k \geq 1$ of the matrix $EH(\theta)$ is continuous with uniformly bounded first and second derivatives on $\theta \in \Theta_n$.
\end{itemize}

\begin{theorem} \label{Thm:consistent}
	Suppose there exist $N>0$ such that $K^\ast_\lambda$ is non-singular with all eigenvalues bounded away from 0 for all $n>N$. Under Conditions \emph{C.1} - \emph{C.6}, if $ r_1{m^\ast}^2 \lambda^{-2}=o_p(1)$, $ r_2\sqrt{{m^\ast}} \lambda^{-1} =o_p(1)$ and $r_3 m^\ast \lambda^{-1}=o_p(1)$, then as $n\to \infty$ we have
	\begin{itemize}
		\item[(i)]$\| \thetahatlambdan-\theta^\ast  \|_1 \overset{P}{\to} 0$, and
		\item[(ii)] $\sqrt{n} {J^\ast_\lambda}^{-\frac{1}{2}} K^\ast_\lambda (\thetahat_{\lambda}-\theta^\ast) \overset{D}{\to} N_p(0,I)$,
	\end{itemize}
	where  $J^\ast_\lambda = Cov \left [u(\theta^\ast,\wstarlambdan) \right ]$ and $K^\ast_\lambda = -E \nabla u(\theta^\ast,\wstarlambdan)$  denote $p \times p$ population variability and sensitivity   matrices.
\end{theorem}

\begin{proof}
	Without loss of generality, we only prove the case $p=1$. Since $p$ is fixed, the proof can be easily generalized to the case $p>1$ without additional conditions. Let $\hat K_\lambda= -\efn \nabla u(\thetahat,\what_\lambda)$ be the empirical sensitivity matrix. Then $\thetahat_\lambda$ can be written as
	$
	\thetahatlambdan =\thetahat +   \hat K_\lambda ^{-1} \efn u(\thetahat,\what_\lambda)
	$, with $\thetahat$ being a consistent preliminary estimator.  Note that  $E u(\theta^\ast,w^\ast_\lambda)=0$ and 
	\begin{align}  
	&\|\efn u(\thetahat, \what_\lambda) - Eu(\theta^\ast, w^\ast_\lambda)\|_1 \leq  \label{eq:f3}
	\| \efn u(\thetahat, \what_\lambda) - \efn u(\thetahat, w^\ast_\lambda)\|_1 \\
	& + \| \efn u(\thetahat, w^\ast_\lambda) -  \efn u(\theta^\ast, w^\ast_\lambda)\|_1 
	+\| \efn u(\theta^\ast, w^\ast_\lambda) - E u(\theta^\ast, w^\ast_\lambda)  \|_1. \notag
	\end{align}
	The first term on the right hand side of (\ref{eq:f3}) is $o_p(1)$ by Lemma \ref{Lem:covergeinuthetahat}. The second term is $o_p(1)$ since $\| \efn u(\thetahat, w^\ast_\lambda) -  \efn u(\theta^\ast, w^\ast_\lambda)\|_1 \leq  \max_j|\efn u_j(\thetahat)-\efn u_j(\theta^\ast)| \| w^\ast_\lambda \|_1 $, which converges to 0 by Theorem's assumptions and Lemma \ref{Lem:wnorm}. The last term of (\ref{eq:f3}) is also $o_p(1)$ by the Law of Large Numbers. This shows that $\efn u(\thetahat, \what_\lambda) \overset{P}{\to} 0$.  Moreover, from Lemma \ref{Lem:dhnbarhnstar}, $ \| \hat K_\lambda - K^\ast_\lambda  \|_1 =o_p(1)$. Since $K^\ast_\lambda$ has all eigenvalues bounded away from 0 for large $n$,  we have $\hat K_\lambda^{-1} \efn u(\thetahat,\what_\lambda) \overset{P}{\to} 0$. Since $\thetahat \overset{P}{\to} \theta^\ast$, we have $
	\thetahatlambdan =\thetahat +   \hat K_\lambda^{-1} \efn u(\thetahat,\what_\lambda) \overset{P}{\to} \theta^\ast
	$, which shows part (i) of the theorem.	
	
	To show normality in (ii), re-arrange
	$
	\thetahatlambdan =\thetahat + \hat K_\lambda^{-1}  \efn u(\thetahat,\what_\lambda)
	$
	and obtain
	\begin{align}	 
	&\hat K_\lambda (\thetahatlambdan- \theta^\ast)  =  \hat K_\lambda (\thetahat-\theta^\ast) +\efn u(\thetahat,\what_\lambda)  
	=  \efn u(\theta^\ast,\what_\lambda) +  [\hat K_\lambda - \widetilde{ K}_\lambda] (\thetahat-\theta^\ast) \notag\\
	=& \efn u(\theta^\ast,w^\ast_\lambda) + [\efn u(\theta^\ast,\what_\lambda) - \efn u(\theta^\ast,w^\ast_\lambda) ] 
	+ [\hat K_\lambda - \widetilde{ K}_\lambda] (\thetahat-\theta^\ast), \label{eq:f4}
	\end{align}
	where $\widetilde K_\lambda= -\efn \nabla u(\widetilde \theta ,\what_\lambda)$ and $\widetilde{\theta}$ is some value between $\thetahat$ and $\theta^\ast$. The second equality follows from the first-order expansion of $\efn u(\theta^\ast,\what)$ at $\thetahat$.  For the first term in (\ref{eq:f4}), we have  $\sqrt{n} {J^\ast_\lambda}^{-\frac{1}{2}} \efn u(\theta^\ast,w^\ast_\lambda) \overset{D}{\rightarrow} N_p(0, I)$, since the Lindeberg-Feller Central Limit Theorem applies to $u(\theta^\ast,w^\ast_\lambda)$ by Lemma \ref{Lem:lindeberg}.  By Lemma \ref{Lem:varuorder} $J^\ast_\lambda=O(m^\ast)$, so the first term in (\ref{eq:f4}) is $O_p(\sqrt{m^\ast/n})$. The second term in (\ref{eq:f4})   $\efn u(\theta^\ast,\what_\lambda) - \efn u(\theta^\ast,w^\ast_\lambda) = \efn U(\theta^\ast)^T (\what_\lambda-w^\ast_\lambda)$  is of smaller order compared to first term
	$\efn u(\theta^\ast,w^\ast_\lambda)=\efn U(\theta^\ast)^T w^\ast_\lambda$ since $\|\what_\lambda - w^\ast_\lambda\|  \overset{P}{\rightarrow} 0$ by Theorem \ref{Thm:unifweights}.   For the last term in (\ref{eq:f4}), we have
	\begin{align}
	|(\hat K_\lambda -\widetilde K_\lambda)(\thetahat -\theta^\ast)|_1 \leq & \left\{ \max_j \left\vert \efn \nabla u_j(\thetahat)-  E\nabla u_j(\theta^\ast) \right\vert \right. \notag \\
	 &+   \left.  \max_j \left\vert \efn\nabla u_j(\widetilde{\theta}) -  E\nabla u_j(\theta^\ast)\right\vert \right\}
	\|\what_\lambda\|_1 |\thetahat-\theta^\ast|  \notag  \\
	\leq 
	& 2 r_2 \|\what_\lambda\|_1 |\thetahat-\theta^\ast|. \label{eq:third_term}
	\end{align}
	and the last expression in (\ref{eq:third_term}) is $o_p(r_2 m^\ast n^{-1/2}\lambda^{-1})$ by Lemma \ref{Lem:wnorm}. 
	Theorem's assumption that $r_1{m^\ast}^2 \lambda^{-2}=o_p(1)$ implies that the last term in (\ref{eq:f4}) is of smaller order compared to the first term. Finally, since $\|\hat K_\lambda -K^\ast_\lambda\|_1=o_p(1)$ according to Lemma \ref{Lem:dhnbarhnstar}, Slutsky's Theorem implies the desired result.      	
\end{proof}

Consistency and asymptotic normality for the one-step estimator $\thetahatlambdan$ follow mainly from $\what_\lambda$ converging in probability to the target composition rule $\wstarlambdan$. Since each sub-likelihood score is unbiased and asymptotically normal, their linear combination is also normally distributed. The overall convergence rate is given by $\|\sqrt{n} {J^\ast_\lambda}^{-\frac{1}{2}} K^\ast_\lambda\|_1$ which is of order between $\sqrt n$ and $\sqrt{nm}$. The actual order depends on the underlying correlation between  partial scores  $u_1, \dots, u_m$. While the optimal rate  $\sqrt{nm}$ is achieved when the scores are perfectly independent, combining highly correlated scores into the final estimating equation will give rates closer to $\sqrt{n}$.

\section{Examples for special families of models}\label{sec:example} In this section, we illustrate the SCLE through estimation of location and scale estimation for special multivariate normal models. 

\subsection{Estimation of common location for heterogeneous variates}  \label{sec:anaexample1}

Let $X \sim N_m(\theta 1_m, \Sigma)$, where the $m \times m$ covariance matrix $\Sigma$ has off-diagonal elements $\sigma_{jk}$ ($j\neq k$) and diagonal elements $\sigma^2_k$ ($j=k$). Computing the MLE of $\theta$  requires   $\Sigma^{-1}$ and in practice $\Sigma$ is replaced by the MLE $\hat{\Sigma} = \efn X^T X$. When $n<m$,   $\hat{\Sigma}$ is singular and the MLE of $\theta$ is not available in practice, whilst CL estimation is still feasible. The $j$th partial score is  $u_j(\theta)=(X_j-\theta)/ \sigma_j^2$ and the CL estimating equation (\ref{estimating_eq}) based on the sample $X^{(1)}, \dots, X^{(n)}$ is
$$
0 =  \efn u(\theta, w) = \sum_{j=1}^m   \dfrac{w_j}{n \sigma_j^{2} } \sum_{i=1}^n (X^{(i)}_j-\theta),
$$
leading to the profiled MCLE  
\begin{equation}\label{eq:estimator_example1}
\hat \theta(w)= \left(\sum_{j=1}^m w_j\sigma_j^{-2}\overline{X}_j \right)/ \left(\sum_{j=1}^m w_k\sigma_j^{-2}\right),
\end{equation}
which is a weighted average of marginal sample means $\overline{X}_j=n^{-1}\sum_{i=1}^n X_j^{(i)}$, $j\ge 1$. In this example, one can work out directly the optimal composition rule $w^\ast_\lambda$ and no estimation is required. Particularly, it is useful to inspect the special case where $X$ has independent components ($\sigma_{jk}=0$ for all $j\neq k$). This corresponds to the fixed-effect meta-analysis model where estimators from  $m$ independent studies are combined to improve accuracy. Under independence, we have the explicit solution
$$
w^\ast_{\lambda,j}= (1-\sigma^2_j \lambda) I\left(\sigma^2_j< \lambda^{-1}\right),  \ \ 1 \leq j \leq m,
$$
which highlights that overly noisy data subsets with variance $\sigma^2_j \geq \lambda^{-1}$ are dropped and thus do not influence the final estimator (\ref{eq:estimator_example1}).  The  number of non-zero elements in $w^\ast_\lambda$ is $\sum_{j=1}^m I (\sigma^2_j< \lambda^{-1} )$. 

Note that when $\lambda = 0$, we have uniform weights $w^\ast_0=(1,\dots,1)^{T}$ and the corresponding MCLE is the usual optimal meta-analysis solution. Although the implied estimator $\thetahat(w^\ast_0)$  has minimum variance, it offers no control for the overall computational cost since all $m$ sub-scores are selected. On the other hand, choosing judiciously $\lambda>0$ may lead to low computational burden with negligible loss for the resulting estimator. For instance, assuming $\sigma^2_j=j^2$, for $\theta \in \Theta_n$, a straightforward calculation shows 
\begin{align} 
E \left[ u(\theta,w^\ast_\lambda)- u(\theta, w^\ast_0) \right]^2
\leq  \lambda^2  \sum_{j \in \Ecal} j^2    +  \sum_{j \notin \Ecal} {j^{-2}}       + o(1),
\end{align}
Since the number of the non-zero scores $\sum_{j=1}^m I\left( j^2< \lambda^{-1}\right) =  \lfloor \lambda^{-\frac{1}{2}} \rfloor$, the first term the mean squared difference between $u(\theta,w^\ast_\lambda)$ and the optimal score $u(\theta, w^\ast_0)$ 
is bounded by $\lambda^2\lambda^{-1}\lambda^{-\frac{1}{2}} =\lambda^{\frac{1}{2}}$, up to a vanishing term. Thus, if $\lambda = o(1)$, the composite score  $u(\theta,w^\ast_\lambda)$  converges to the optimal composite score $u(\theta,w^\ast_0)$. Particularly, if $\lambda$ decreases  at a sufficiently slow rate, the truncated score $u(\theta,w^\ast_\lambda)$ can still contain a relatively small number of terms, while the correspondent estimator $\thetahat(w^\ast_\lambda)$ is approximately equal the optimal estimator $\thetahat(w^\ast_0)$ in terms of statistical accuracy.    

If the elements of  $X$ are correlated ($\sigma_{jk}\neq 0$ for $j\neq k$), the partial scores contain overlapping information on $\theta$. In this case, tossing away some highly correlated partial scores improves computing while maintaining satisfactory statistical efficiency for the final estimator.  Figure \ref{fig:anaexample1} shows the solution path of $w^\ast_\lambda$ and the asymptotic relative efficiency of the  corresponding SCLE $\thetahat(w^\ast_\lambda)$ compared to MLE for different values of $\lambda$. When $m$ is large (e.g. $m=1000$), the asymptotic relative efficiency drops gradually until a few scores are left.  This example illustrates that a relatively high efficiency can be achieved by our truncated CL equations, when a few partial scores  already contains the majority of information about $\theta$. In such cases, the final SCLE with a sparse composition rule is expected to achieve a good trade-off between computational cost and statistical efficiency.

\begin{figure}[htp]
	\centering
	\quad\quad\quad\quad$\rho=0$, $m=20$ \quad\quad\quad\quad\quad\quad $\rho=0.5$, $m=20$ \quad\quad\quad\quad\quad\quad $\rho=0.5$, $m=1000$\\
	\quad\;\; Number of sub-likelihoods \quad\; Number of sub-likelihoods \quad Number of sub-likelihoods
	\includegraphics[scale=0.86]{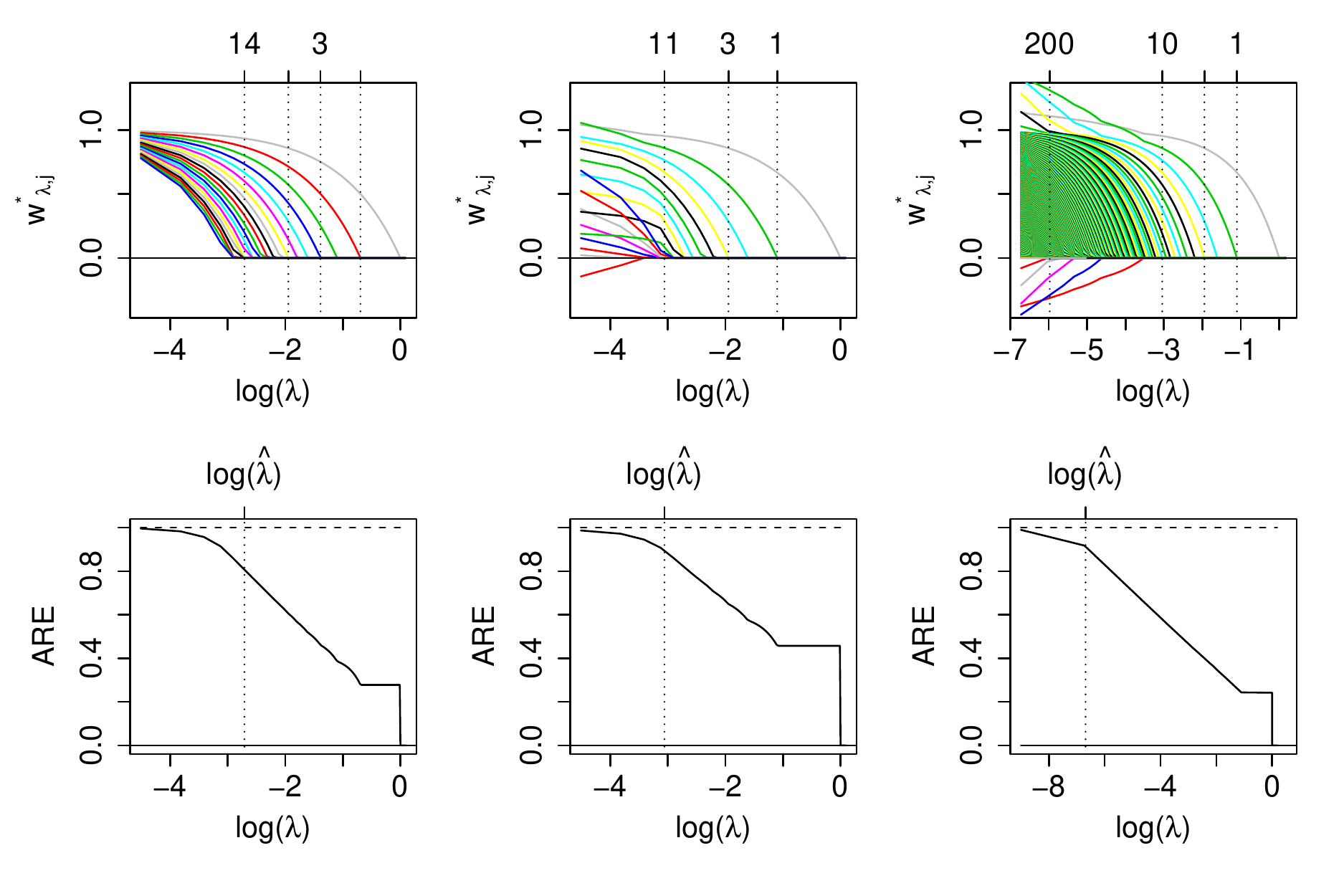}
	\caption{Top Row: Solution paths for the minimizer $w^\ast_\lambda$ of Criterion $Q(\theta^\ast,w)$ in (\ref{eq:criterion}) for different values of $\lambda$ with corresponding number of sub-likelihoods. Bottom Row: Asymptotic relative efficiency (ARE) of the SCLE $\thetahat(w^\ast_\lambda)$ compared to MLE.  The vertical dashed lines on the bottom represent  $\hat\lambda$ selected by Criterion (\ref{eq:lambda}) with $\tau=0.9$. Results correspond to the common location model $X\sim N_m(\theta^\ast 1_m, \Sigma)$ with $j$th diagonal element of $\Sigma$ equal to $j$, and $(jk)$th off-diagonal element of $\Sigma$ equal to $\rho\sqrt{jk}$.   }
	\label{fig:anaexample1}
\end{figure}

\subsection{Location estimation in exchangeable normal variates} \label{sec:anexample2}

In our second example we consider exchangeable variables with $X \sim N_m(\theta 1_m, \Sigma)$ with ${\Sigma =  
	(1 -  \rho) I_m + \rho 1_m 1^T_m }$, $0<\rho<1$. The marginal scores  $u_j(\theta)=X_j-\theta$ are identically distributed and exchangeable with equal correlation. Differently from Example \ref{sec:anaexample1}, the solution $w^\ast_{\lambda}$ to Criterion (\ref{eq:criterion}) has equal elements
$$
w^\ast_{\lambda,j}= \dfrac{1-\lambda}{\rho(m -1) + 1} I(\lambda<1), \ \ 1 \le j \le m,
$$
so the optimal parameter estimator is $ \thetahat(w^\ast_{\lambda})= \sum_{j=1}^m \overline{X}_j/m$  regardless of the value of $\lambda$.  The first eigenvalue of $ES(\theta)$ is $\theta(m-1)+1$, whilst the remaining $m-1$ eigenvalues are all equal to $1-\theta$, suggesting that the first score contains a relatively large information on $\theta$ compared to the other scores.  When $m$ is much larger than $n$, we have 
$
var \{\thetahat(w^\ast_{0}) \} = [\rho^2(m-1)+1]/(m n) \asymp  \rho^2/n
$. The trade-off between statistical and computational efficiency may be measured by the ratio of estimator's variance with  $m=  \infty$ compared to that with  $m<\infty$. This ratio is $t(m) =
\rho^2 m / \{\rho^2(m-1)+1\}$, 
which increases quickly for smaller $m$ and much slower for larger $m$ (e.g., $t(5) = 0.83$, $t(9) = 0.90$ and $t(50) = 0.98 $, if $\rho=0.75$). Thus, although all the elements in $w^\ast_\lambda$ are nonzero, a few partial scores contain already the majority of the information on $\theta$. This suggests that in practice taking a sufficiently large value for $\lambda$,  so that the sparse empirical solution $\what_{\lambda}$ contains only a few of zero elements, already ensures a relatively high statistical efficiency for the corresponding MLCE $\thetahat(\what_\lambda)$.

\subsection{Exponentially decaying covariances}\label{sec:anaexample3}

Let $X \sim N_d(0, \Sigma(\theta))$, where the $jk$th element of $\Sigma(\theta)$ is $\sigma_{jk}(\theta) = \exp\{-\theta d(j,k)\}$. The quantity  $d(j,k)$ may be regarded as the distance between spatial locations $j$ and $k$. Evaluating the ML score in this example is computationally expensive when $d$ is large, since it requires computing the inverse  of $\Sigma(\theta)$, a task involving $O(d^3)$ operations. On the other hand, the CL score is obtained by inverting  $2\times2$ covariance matrices, thus requiring at most $O(d^2)$ operations. Given i.i.d. observations   $X^{(1)},\dots,X^{(n)}$ on $X$, the MCLE $\hat \theta(w)$ solves the  equation
\begin{align*}  
0 = & \sum_{j<k} w_{jk}   \sum_{i=1}^n   u_{jk}(\theta, X_j^{(i)},X_k^{(i)})    \\  
=& \sum_{j<k} w_{jk}   \sum_{i=1}^n   \left[    \dfrac{\sigma_{jk}(\theta)\{ {X^{(i)}_j}^2+{X^{(i)}_{k}}^2-2 X^{(i)}_jX_k^{(i)}\sigma_{jk}(\theta)\}}{\{1-\sigma_{jk}(\theta)^2\}^2} \right] \sigma_{jk}(\theta)d(j,k) \notag\\
&- \sum_{j<k} w_{jk}   \sum_{i=1}^n   \left[\frac{\sigma_{jk}(\theta)+X^{(i)}_jX_k^{(i)}}{ 1-\sigma_{jk}(\theta)^2} 
\right]  \sigma_{jk}(\theta)d(j,k) , \notag
\end{align*}
where $u_{jk}$ corresponds to the score of a bivariate normal distribution for the pair $(X_j, X_{k})$.  

Figure \ref{fig:anaexample3} shows the analytical solution path of the minimizer $w^\ast_\lambda$ of Criterion (\ref{eq:criterion}) for different values of $\lambda$, and the asymptotic relative efficiency of the SCLE $\hat{\theta}(w^\ast_\lambda)$ compared to MLE.  We consider a number of pairs ranging from  $m=45$ to $m=1225$  for various choices of $\theta$. When $\lambda=0$, the SCLE has relatively high asymptotic efficiency.  Interestingly, efficiency remains steady around $90\%$ until only a  few sub-likelihoods are left. This suggests again that a very small proportion of partial-likelihood components contains already the majority of the information about $\theta$. In such cases, the SCLE reduces dramatically the computing burden while retaining satisfactory efficiency for the final estimator.

\begin{figure}[htp]
	\centering
	\quad\quad$\theta=0.4$, $d=10$, $m=45$   \quad\quad\quad $\theta=0.6$, $d=10$, $m=45$ \quad $\theta=0.6$, $d=50$, $m=1225$\\
	\quad\;\; Number of sub-likelihoods \quad\; Number of sub-likelihoods \quad Number of sub-likelihoods
	\includegraphics[scale=0.86]{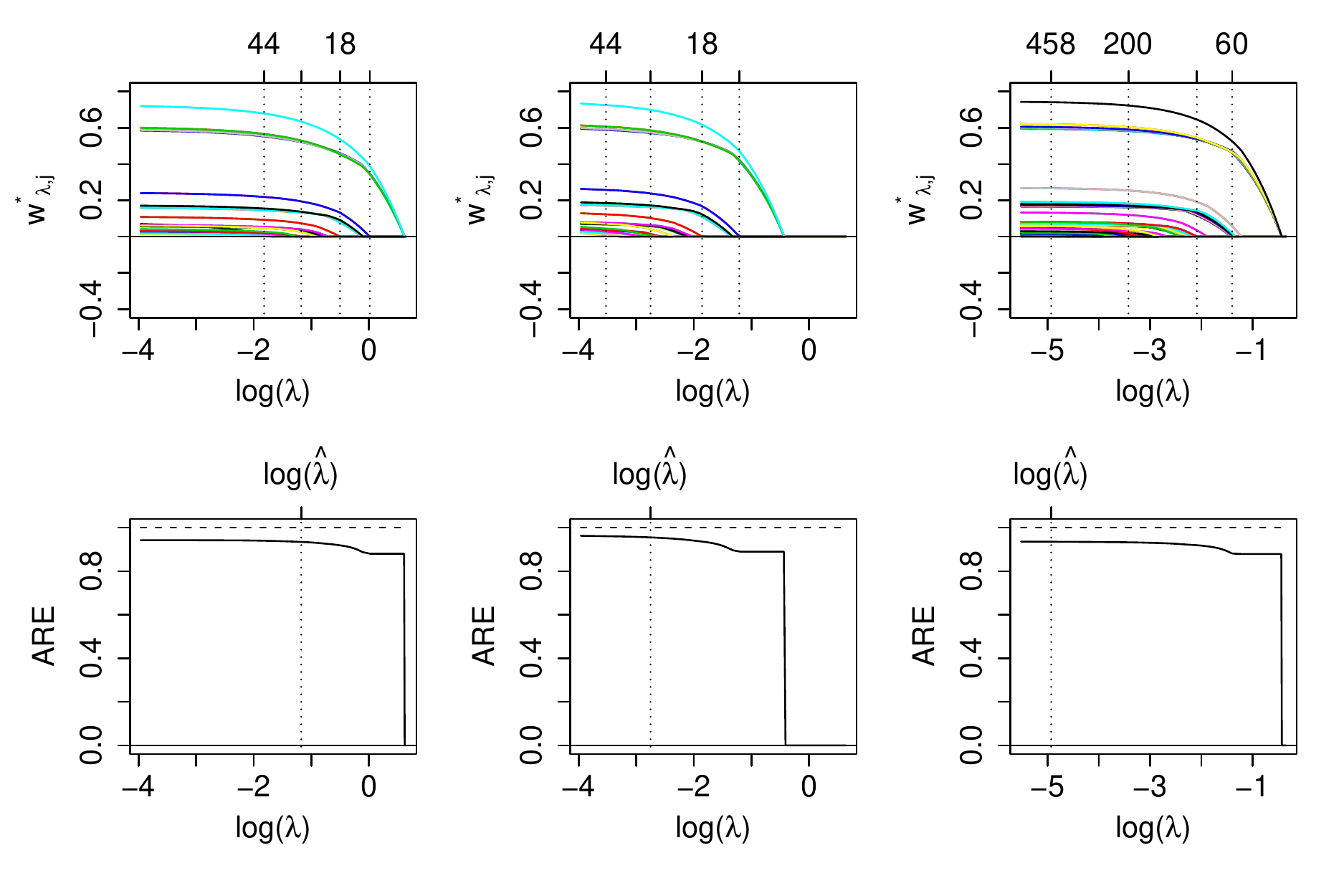}
	\caption{Top Row: Solution paths for the minimizer $w^\ast_\lambda$ of Criterion $Q(\theta^\ast,w)$ defined in (\ref{eq:criterion}) for different values of $\lambda$ with corresponding number of sub-likelihoods reported. Bottom Row: Asymptotic relative efficiency (ARE) of SCLE $\thetahat(w^\ast_\lambda)$ compared to MLE.   The vertical dashed lines on the bottom row correspond to $\hat\lambda$ selected by Criterion (\ref{eq:lambda}) with $\tau=0.9$. Results correspond to the model $X\sim N_d(0, \Sigma(\theta))$ with $(j,k)$th  element of $\Sigma$ equal to $\exp\{ -\theta \sqrt{2|j-k|} \}$. }
	\label{fig:anaexample3}
\end{figure}

\section{Numerical examples} \label{sec:NumExample} In this section, we study the finite-sample performance of the SCLE in terms by assessing its mean squared error and computing cost when the data dimension $d$ increases. As a preliminary estimator, we use the MCLE $\hat \theta(w)$ with  $w =(1,\dots, 1)^T$, which is perhaps the most common choice for $w$ in CL applications  \citep{varin2011overview}.

\subsection{Example 1} We generate samples of size $50$ from  $X\sim N_m(\theta1_m, \Sigma_l)$, $l=1,\dots, 4$. We specify the following covariance structures:  $\Sigma_1=I_m$; $\Sigma_2$ is diagonal with $k$th diagonal elements $\sigma_k^2=k$; $\Sigma_3$ has unit diagonal elements with the first 10 elements of $X$ uncorrelated with any other element while the other elements in $X$ have  pairwise correlations $0.8^{|j-k|}$ ($10 < j <k<d$);  $\Sigma_4$ has unit diagonal elements and  a block diagonal structure with  independent blocks of six elements each and within-block correlation of 0.6.

Figure \ref{fig:example2_1} (left),  shows the relative mean squared error of the SCLE $\thetahat_\lambda$ compared to that of the MLE for a moderate data dimension ($d=m=30$). The points in the trajectories correspond to inclusion of a new sub-likelihood component according to the the least-angle algorithm described in Section \ref{Sec:lambda}. The SCLE $\hat{\theta}_{\lambda}$ achieves more than 90$\%$ efficiency compared to MLE for all the covariance structures considered, always before all the candidate partial likelihoods are included. The advantage of SCLE becomes evident when the sub-likelihood scores exhibit relatively strong correlation. For example, for $\Sigma=\Sigma_4$ where sub-likelihoods  are independent between blocks, the maximum efficiency is achieved when only a few representative partial scores are selected from each block.  

Figure \ref{fig:example2_1} (right) shows the ratio between mean squared error of the SCLE compared and that of the MLE for a relatively large data dimension $(d=m=1000)$ compared to the sample size  ($n=50$). Although here the MLE is used as a theoretical benchmark, in practice such an estimator is not available as $m$ is larger than the sample size $n$. Interestingly, when the sample size $n$ is fixed, including all the sub-likelihoods eventually leads to substantial loss of efficiency.  In this examples, selecting too many sub-likelihoods not only wastes computing resources but also implies estimators with larger errors . On the other hand, a proper choice of the tuning constant $\lambda$ (corresponding to about 20 selected sub-likelihoods) can balance computational and statistical efficiency.

\begin{figure}[h]
	\begin{tabular}{cc}
		\includegraphics[scale=1]{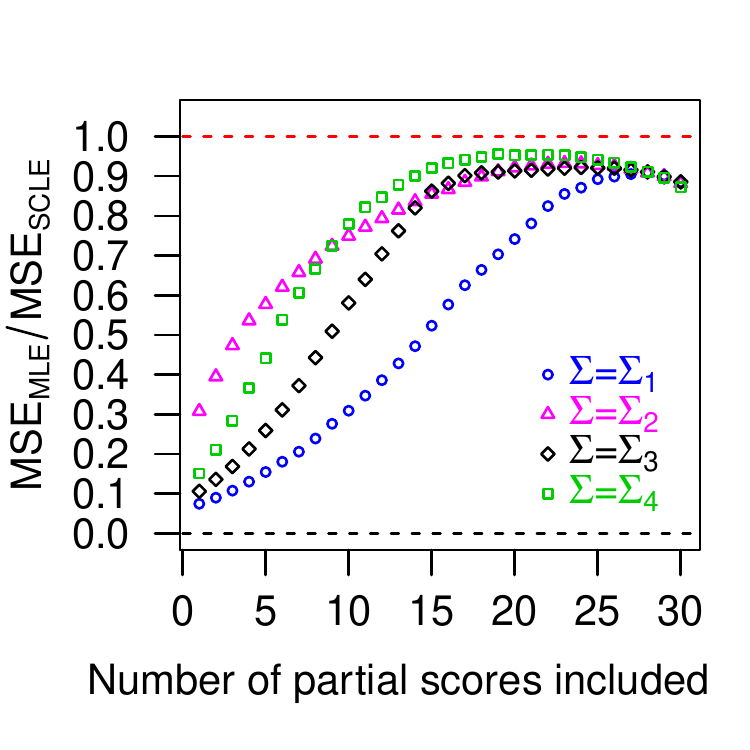} & \includegraphics[scale=1]{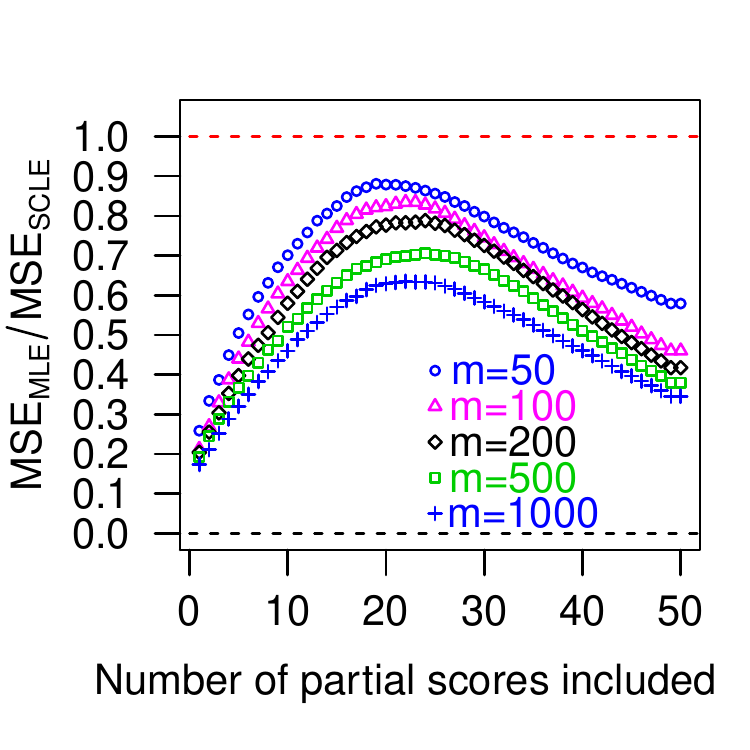} 
	\end{tabular}
	\caption{Monte-Carlo estimate of the mean square error of the MLE ($\text{MSE}_{\text{MLE}}$) divided by that of the SCLE ($\text{MSE}_{\text{SCLE}}$), for the model $X\sim N_m(\theta 1_m, \Sigma)$. Each trajectory is based on 1000 Monte-Carlo samples of size $n=50$. Each point in the trajectories correspond to inclusion of a new sub-likelihood component based on the least-angle algorithm described in Section \ref{Sec:lambda}. Left: Different specifications for $\Sigma$ detailed in Section \ref{sec:NumExample} with $m=30$.  Right: Covariance $\Sigma=\Sigma_2$  with  $m$ ranging from 50 to 1000.  }
	\label{fig:example2_1}
\end{figure}

\subsection{Example 2}

In our second numerical example, we consider  covariance estimation for the model $X\sim   N_d(0, \Sigma(\theta))$ with $\Sigma(\theta)_{j,k}= \exp\{-\theta 2(j-k)^2\}$. Here the covariance between components $X_j$ and $X_{k}$ in the random vector $X$ decreases rapidly as the distance $(j-k)^2$ between  components of $X$ increases. Figure \ref{fig:example2_2} shows  Monte-Carlo estimates for the mean square error of the SCLE $\thetahat(\what_\lambda)$ compared to that of the MCLE  with uniform composition rule (i.e. $\thetahat(w)$ with $w=(1,\dots,1)^T$), for $\theta=0.2$ and $0.4$. Each point in the trajectories correspond to inclusion of a new sub-likelihood component using the least-angle algorithm described in Section \ref{Sec:lambda}. The SCLE is already more efficient than the uniform MCLE when a handful of partial scores are selected. For example if $\theta=0.2$ and $m=1035$, selecting ten sub-likelihoods already ensures 1.5 times the accuracy of the uniform MCLE. Since the uniform MCLE uses all the $m=1035$ pairs of sub-likelihoods, the SCLE obtains more accurate results at a much lower computing cost.

\begin{figure}[h]
	\begin{tabular}{cc}
		\includegraphics[scale=1]{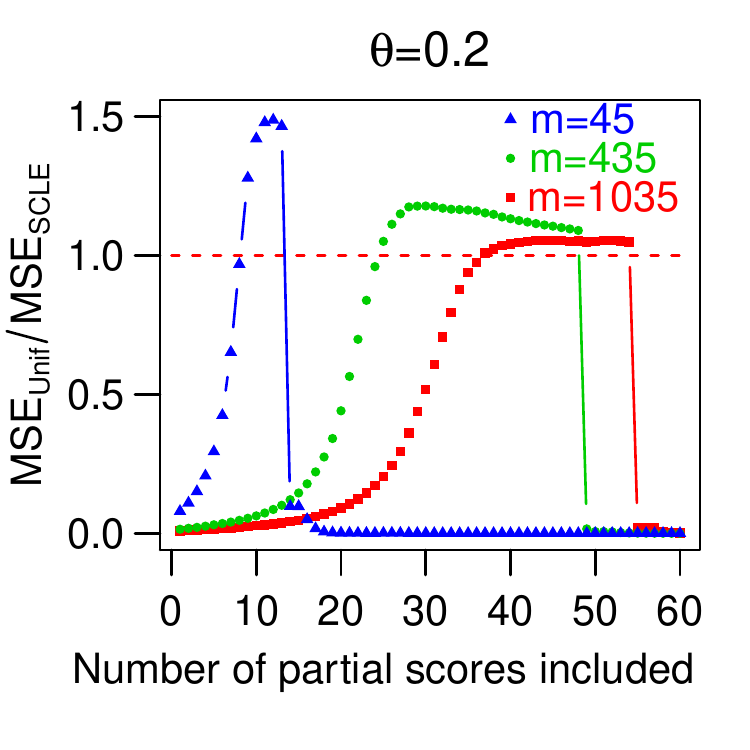} & \includegraphics[scale=1]{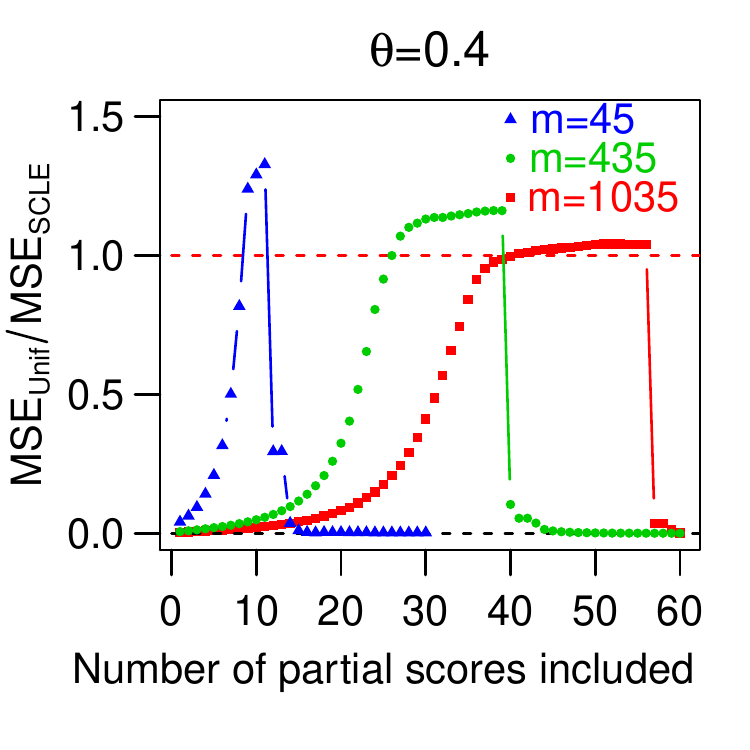} 
	\end{tabular}
	\caption{Monte-Carlo estimate of the mean square error of the MCLE  with $w=(1,\dots, 1)^T$ ($\text{MSE}_\text{Unif}$) divided by that for the SCLE ($\text{MSE}_\text{SCLE}$). Each point in the trajectories corresponds to the inclusion of a new sub-likelihood component based on the least-angle algorithm described in Section \ref{Sec:lambda}. Results are based on 1000 Monte Carlo samples of size $n=50$ from the model $X\sim N_d(0, \Sigma(\theta))$ with $\Sigma(\theta)_{j,k}=\exp\{-2\theta(j-k)^2\}$.  Trajectories correspond to  $\theta=0.2$ (left) and $\theta=0.4$ (right) for different numbers of sub-likelihoods, $m$, ranging from 45 to 1035. }
	\label{fig:example2_2}
\end{figure}

\section{Conclusion and final remarks} \label{sec:finalremarks}

In recent years, inference for complex and large data sets has become one of the most active research areas in statistics. In this context, CL inference has played an important role in applications as a remedy to the drawbacks of traditional likelihood approaches. Despite the popularity of CL methods,  how to address the trade-off between computational parsimony and statistical efficiency in CL inference from a methodological perspective remains a largely unanswered question. Motivated by this gap in the literature, we introduced a new likelihood selection methodology which is able truncate quickly overly complex CL equations potentially encompassing many terms, while attaining relatively low mean squared error for the implied estimator. This is achieved by  selecting CL estimating equations satisfying a $\ell_1$-constraint on the CL complexity while minimizing an approximate $\ell_2$-distance from the full-likelihood score. Inference based on statistical objective functions with $\ell_1$-penalties on the parameter $\theta$ is not new in the statistical literature (e.g., see \cite{giraud2014introduction} for a book-length exposition on this topic). Note, however, that differently from existing approaches the main goal here is to reduce  the computational complexity of the overall CL estimating equations regardless of the model parameter $\theta$, which is viewed as fixed in size. Accordingly, our $\ell_1$-penalty involves only the composition rule $w$, but not the model parameter $\theta$. In the future, developing approaches for simultaneous penalization on $\theta$ and $w$  may be useful to deal with situations where both the data dimension and the size of the parameter space increase.

Two main perks of  the proposed approach make it an effective alternative to traditional CL estimation from  practitioner's perspective. The first advantage is that  the SCLE methodology constructs CL equations and returns inferences very quickly. Theorem \ref{Thm:uniqueroots} shows that for any $\lambda>0$ the empirical composition rule $\whatlambdan$  retains at most $np \wedge m$ non-zero elements. This is an important feature of our method, which reduces -- sometimes dramatically -- the amount of computing needed to obtain the implied MCLE $\thetahat(\whatlambdan)$ and its standard error. Lemma \ref{Lem:KKT} highlights that the non-zero elements in $\whatlambdan$  correspond to partial scores  maximally correlated with the residual difference $r(\theta, w)= u^{ML}(\theta)-u(\theta, w)$. This means that our approach constructs estimators with relatively high efficiency by dropping only those $u_j$s contributing the least in the CL equations for approximating $u^{ML}(\theta)$. The second desirable feature of our method concerns model selection and the ability to reduce the complexity of large data sets. In essence, the truncation step (T-Step) described in  (\ref{eq:LAstep}) is a dimension-reduction step: starting from observations on a possibly large the $d$-dimensional vector $X$, our method generates a collection of lower-dimensional subsets $\mathcal{S}_\lambda = \{S_j, j \in   \Ecalhat_\lambda, \lambda >  0 \}$  where $\Ecalhat_\lambda = \{ j: \hat w_{\lambda,j} \neq 0 \}$. While individually the selected data subsets in  $\mathcal{S}_\lambda$ are of size much smaller than $d$, collectively they contain most of the information on $\theta$ for a given level of computing represented by $\lambda$. 

From a theoretical perspective,  little work has been devoted to study the properties of CL estimators when the number of sub-likelihoods $m$ diverges. \cite{cox2004note} discuss estimators based CL equations with $m= {{d}\choose{2}} + d$ terms by taking all pairwise and marginal scores for the $d$-dimensional vector $X$. They take non-sparse and more rigid composition rules compared to ours with $w_{jk}=1$ for all pairs ($j\neq k$) and $w_{jj}= -a \times d$ for all marginals ($j=k$), where $a$ is a tuning constant used to increase efficiency. To our knowledge, the current paper is the first studying the behavior of more flexible sparsity-inducing composition rules and implied CL estimating equations in the setting where both $m$ and $n$ grow. 

Theorem \ref{Thm:unifweights} and Corollary \ref{Col:covergeinu} provide us with guidance on when the selected score $u(\theta, \hat w_\lambda)$  is  a meaningful approximation to the unknown ML score in the sense of the objective (\ref{eq:criterion}). A first requirement is that the total information on $\theta$ available if the full likelihood were actually known, $m^\ast= \Vert u^{ML}(\theta^\ast) \Vert^2_2$, is not  overwhelming compared to the sample size $n$. If $X \sim N_m(\theta 1_m, \Sigma)$, we require $m^\ast/n  = \text{tr}\{\Sigma^{-1}\}/n \to 0$. This condition is very mild when relatively few elements  $X$ contain a strong signal on $\theta$, whilst the remaining elements  are noisy and with heterogeneous variances. In Section \ref{sec:anaexample1}, we illustrate this by taking $\Sigma$  diagonal with increasing diagonal elements. A second requirement is that the tuning constant $\lambda$ dominates asymptotically $m^\ast \sqrt{r_1}$, where $\sqrt{r_1}$ represents the convergence rate of the empirical covariance of scores $\efn\{S(\theta)\}$. For instance, if the elements of $S(\theta)$ are sub-Gaussian, we have $r_1 =o_p( \log(m)/n)$, meaning that  $\lambda$ should be asymptotically larger than  $m^\ast \sqrt{\log(m)/n}$. 

Finally, we show that statistical optimality and computationally parsimony can co-exhist within the same selection procedure when $\lambda$ is judiciously selected. If $\lambda \to 0$ at the rate described in Theorem \ref{Thm:unifweights}, the truncated composition rule $\hat w_\lambda$ with $|\Ecalhat_\lambda|$ scores approximates the optimal composition rule $w^\ast_0$ consisting of $m$ nonzero terms. Accordingly, Corollary \ref{Col:covergeinu} suggests that the implied truncated CL score function $u(\theta, \hat w_\lambda)$ approximates the optimal score $u(\theta, \hat w_\lambda)$, uniformly on a neighborhood of $\theta^\ast$. Extending this type of result and developing further theoretical insight on the interplay between the type of penalty and the MCLE accuracy beyond the current i.i.d. setting would represent another exciting future research direction. For example, findings would be  particularly valuable in spatial statistics where often the  number of sub-likelihood components is overwhelming and poses serious challenges to traditional CL methods.

\section*{Appendix} In this section, we show technical lemmas required by the main results in Section \ref{sec:properties}.
\addtocounter{section}{1}

\begin{lemma}\label{Lem:decreasingw}
	$\| \what_\lambda \|_1$ and $\|w^\ast_\lambda\|_1$ are decreasing in $\lambda$.
\end{lemma}
\begin{proof}  
	Denote the first term of Criterion $Q_\lambda(\theta,w)$ defined in (\ref{eq:criterion}) (without the penalty term) by $Q_1(\theta,w)$. Suppose $\lambda_1>\lambda_2$, and let $w_1$, $w_2$ be the minimizers of $Q_{\lambda_1}(\theta^\ast,w)$, $Q_{\lambda_2}(\theta^\ast,w)$ respectively. Then,
	$ Q_1(w_1)+\lambda_1\|w_1\|_1 \leq Q_1(w_2)+\lambda_1\|w_2\|_1$ and $Q_1(w_1)+\lambda_2\|w_1\|_1 \geq Q_1(w_2)+\lambda_2\|w_2\|_1$. 
	Subtracting the last two inequalities gives $(\lambda_1-\lambda_2)\| w_1\|_1 \leq (\lambda_1-\lambda_2)\| w_2\|_1 $. Since $\lambda_1>\lambda_2$, we have $\| w_1\|_1 \leq\| w_2\|_1$. An analogous argument shows that $\| \what_\lambda \|_1$ is decreasing.\end{proof}

\begin{lemma}\label{Lem:wnorm}
	Under Conditions \emph{C.1} - \emph{C.3}, if $r_1/\lambda = o_p(1)$, then $\| w^\ast_\lambda \|_1=O({m^\ast}/\lambda)$ and $\| \what_\lambda \|_1=O_p({m^\ast}/\lambda)$.
\end{lemma}

\begin{proof} 
	For $w^\ast_\lambda$,  note that $Q_\lambda(\theta^\ast,w^\ast_\lambda)= E\| u^{ML}(\theta^\ast) - u(\theta^\ast, w^\ast_\lambda) \|_2^2/2 + \lambda\| w^\ast_\lambda \|_1\leq Q_\lambda(\theta^\ast, 0) = {m^\ast}/2$, or   $\lambda \| w^\ast_\lambda\|_1 \leq  {m^\ast}/2$. Hence, $\| w^\ast_\lambda\|_1 = O({m^\ast}/\lambda)$. For $\what_\lambda$,  we have  $\hat Q(\thetahat,\what_\lambda)=  \what_\lambda^T \efn S(\thetahat) \what_\lambda/2 - \text{diag}(\efn S(\thetahat))^T \what_\lambda +\lambda \| \what_\lambda \|_1 \leq \hat Q(\thetahat,0) =0$.  Since $Eu^{ML}(\theta^\ast)^T u^{j}(\theta^\ast)=Eu_j(\theta^\ast)^Tu_j(\theta^\ast)$, we have
	\begin{align*}
	\lambda \| \what_\lambda \|_1   \leq  \text{diag}(\efn S(\thetahat))^T \what_\lambda 
	 =\text{diag}( \efn S(\thetahat) -E S(\theta^\ast) )^T \what_\lambda + E u^{ML}(\theta^\ast)^T u(\theta^\ast,\what_\lambda),
	\end{align*}
	with $\text{diag}( \efn S(\thetahat) -E S(\theta^\ast) )^T \what_\lambda  \leq \max_{j,k} |\efn S(\thetahat) -E S(\theta^\ast) )^T|_{j,k}  \| \what_\lambda \|_1     \leq r_1 \| \what_\lambda \|_1$ and 
	$$ 
	E u^{ML}(\theta^\ast)^T u(\theta^\ast,\what_\lambda)=O_p({m^\ast})
	$$
	Hence, $\| \what_\lambda  \|_1=O_p({m^\ast}/\lambda)$. \end{proof}

\begin{lemma} \label{Lem:covergeinuthetahat}  Let $\lambda \to 0$ as $n \to \infty$. Under Conditions \emph{C.1}-\emph{C.3}, if $r_1{m^\ast}^2\lambda^{-2}=o_p(1)$, we have $ \efn  \left 	\|   u(\thetahat,\what_\lambda)-    u(\thetahat,\wstarlambdan) \right\|_2 \overset{P}{\to} 0$,  as $n \to \infty$, where $\thetahat$ is the preliminary root-$n$ consistent estimator used to compute $\whatlambdan$ in the T-Step (\ref{eq:LAstep}).
\end{lemma}
\begin{proof}
	Note that $\frac{r_1{m^\ast}^2}{\lambda^2}=o_p(1)$ implies $r_1/\lambda=o_p(1)$. Therefore, we have $\| \what_\lambda-w^\ast_\lambda\|_1=O_p({m^\ast}/\lambda)$ by Lemma \ref{Lem:wnorm}. Moreover, re-arranging $\widehat{Q}_{\lambda}(\thetahat,\what_\lambda)\leq \widehat{Q}_{\lambda}(\thetahat,w^\ast_\lambda)$ gives
	\begin{align*}
	\frac{1}{2} \efn\{\what_\lambda^T S(\thetahat)\what_\lambda -  \wstarlambdant  S(\thetahat) w^\ast_\lambda  \} 
	\leq 
	\efn\{{U^2}(\thetahat)^T(\what_\lambda-w^\ast_\lambda)\}
	-\lambda\|\what_\lambda\|_1+\lambda\|w^\ast_\lambda\|_1.
	\end{align*}
	Subtracting $\efn\{U(\thetahat)U(\thetahat)^T w^\ast_\lambda\}^T(\what_\lambda-w^\ast_\lambda)$ from both sides  gives
	\begin{align*} 
	& \frac{1}{2} \efn\|U(\thetahat)^T(\what_\lambda-w^\ast_\lambda)\|^2_2
	\leq    \efn\{c(\thetahat,w^\ast_\lambda)\}^T (\what_\lambda-w^\ast_\lambda) - \lambda\|\what_\lambda\|_1+\lambda\|w^\ast_\lambda\|_1\\
	=&\left[   \efn c(\thetahat,w^\ast_\lambda)- Ec(\theta^\ast,w^\ast_\lambda)  \right]^T \what_\lambda - \left[   \efn c(\thetahat,w^\ast_\lambda)- Ec(\theta^\ast,w^\ast_\lambda)  \right]^T w^\ast_\lambda \\
	&+ \left\{    Ec(\theta^\ast,w^\ast_\lambda)^T\what_\lambda -\lambda\|\what_\lambda \|_1    \right \}  -  \left\{    Ec(\theta^\ast,w^\ast_\lambda)^Tw^\ast_\lambda -\lambda\|w^\ast_\lambda \|_1    \right \}\\
	\leq&  \left[   \efn c(\thetahat,w^\ast_\lambda)- Ec(\theta^\ast,w^\ast_\lambda)  \right]^T (\what_\lambda- w^\ast_\lambda)\\
	=& \left\{  \text{ diag}\left(  \efn  S(\thetahat) - E S(\theta^\ast)    \right )    -    \left[  \efn S(\thetahat)) - E  (S(\theta^\ast)   \right ]^T w^\ast_\lambda    \right\}(\what_\lambda- w^\ast_\lambda),
	\end{align*}
	where the inequality is implied by Lemma \ref{Lem:KKT}. The last expression is $o_p(1)$, since $r_1{m^\ast}^2/\lambda^2=o_p(1)$ and $\| \what_\lambda-w^\ast_\lambda\|_1=O_p({m^\ast}/\lambda)$ by Lemma \ref{Lem:wnorm}, and the matrix maximum norm is bounded by matrix 2-norm.      \end{proof}

\begin{lemma}\label{Lem:dhnbarhnstar}
	If  $r_1{m^\ast}^2 \lambda^{-2}=o_p(1)$ and $r_2\sqrt{{m^\ast}} \lambda^{-1} =o_p(1)$, then under conditions \emph{C.1-C.6}, $ \|  \hat K_\lambda- K^\ast_\lambda   \|_1 =o_p(1)$. 
\end{lemma}

\begin{proof}
	This is a direct result since $\|\hat{K}_\lambda - K^\ast_\lambda \|_1 \leq r_2 \|\what-w^\ast \|_1$, $r_2 \overset{P}{\to} 0$ according to lemma assumption and $\|\what-w^\ast \|_1 \overset{P}{\to} 0$  by Theorem \ref{Thm:unifweights}. \end{proof}

\begin{lemma}\label{Lem:varuorder}
	Under Conditions \emph{C.1-C.6},  $E\| u(\theta^\ast,w^\ast_\lambda) \|^2_2 = O({m^\ast})$. 
\end{lemma}

\begin{proof}
	Note that $E\| u^{ML}(\theta^\ast) - u(\theta^\ast, w^\ast_\lambda) \|^2_2 \leq  E \|u^{ML}(\theta^\ast )- u(\theta^\ast,w^\ast_\lambda)\|^2_2 + \lambda \| w^\ast_\lambda\|_1 \leq  E \|u^{ML}(\theta^\ast )\|^2_2 ={m^\ast}$.  Expanding $E\| u^{ML}(\theta^\ast) - u(\theta^\ast, w^\ast_\lambda) \|^2_2$ gives
	\begin{align*}
	\begin{split}  
	E\|  u(\theta^\ast, w^\ast_\lambda) \|^2_2 &\leq   2Eu^{ML}(\theta^\ast)^T u(\theta^\ast, w^\ast_\lambda) \\
	&\leq   2 \sqrt{E\|  u^{ML}(\theta^\ast ) \|^2_2  \cdot E\|  u(\theta^\ast, w^\ast_\lambda) \|^2_2  }  
	=   2 \sqrt{{m^\ast}}\sqrt{  E\|  u(\theta^\ast, w^\ast_\lambda) \|^2_2  }.
	\end{split} 
	\end{align*}
	Re-arranging gives $E\| u(\theta^\ast,w^\ast_\lambda) \|^2_2 \leq 4{m^\ast}$.    \end{proof}

\begin{lemma}\label{Lem:lindeberg}
	Assume Conditions \emph{C.1-C.6}. For every $\epsilon>0$, we have
	\begin{align*}
	\begin{split}  
	\frac{1}{n J^\ast_\lambda} \sum_{i=1}^n E\left\{    u_i(\theta^\ast,w^\ast_\lambda)^2 I(| u(\theta^\ast, w^\ast_\lambda) |  \geq \epsilon \sqrt{n J^\ast_\lambda})     \right\} \to 0,  \quad \text{as }  n \to \infty, 
	\end{split} 
	\end{align*}
	where $u_i(\theta, w) = \sum_{j=1}^m w_j \nabla \log f_j(X_j^{(i)}; \theta)$ is the composite likelihood score corresponding to the $i$th observation.
\end{lemma}

\begin{proof} Without loss of generality, assume $p=1$.  Recall that $J^\ast_\lambda = E   u(\theta^\ast,w^\ast_\lambda)^2$.  For every $\epsilon>0$, and  constants $a,b>1$ such that $1/a + 1/b =1$ 
	\begin{align} 
	&\frac{1}{n J^\ast_\lambda} \sum_{i=1}^n E\left\{    u_i(\theta^\ast,w^\ast_\lambda)^2 I(| u(\theta^\ast, w^\ast_\lambda) |  \geq \epsilon \sqrt{n J^\ast_\lambda})     \right\}\notag \\
	=& \frac{1}{J^\ast_\lambda}  E\left\{    u_i(\theta^\ast,w^\ast_\lambda)^2 I(| u(\theta^\ast, w^\ast_\lambda) |  \geq \epsilon \sqrt{n J^\ast_\lambda})     \right\} \notag \\
	\leq &  \frac{1}{J^\ast_\lambda}   E\left\{   | u_i(\theta^\ast,w^\ast_\lambda)|^{2a} \right \}  \cdot \left\{  \frac{1}{\epsilon^2 n  } \right \}^\frac{1}{b} = \dfrac{\left( J^\ast_\lambda \right)^{a-1}}{(\epsilon^2 n )^{1/b}}, \label{eq:lind}
	\end{align}
	where the inequality follows by applying  H\"{o}lder's and Chebyshev's inequalities. By the assumption at the beginning of Section \ref{sec:properties}  that $m^\ast = o(\log(n))$, Lemma \ref{Lem:varuorder} implies  $J^\ast_\lambda=o(\log(n))$. Hence,  (\ref{eq:lind}) converges to $0$ as $n \to \infty$, which proves the desired result. \end{proof}

\bibliographystyle{abbrvnat} 

\bibliography{Bibliography}

\end{document}